\documentclass{article}
\usepackage{tikz}
\usepackage{url}

\usepackage[top=1in,bottom=1in,right=1in,left=1in]{geometry}

\usepackage{graphicx}
\usepackage{amsthm}
\usepackage{amsfonts}
\usepackage{amssymb}
\usepackage{amsmath}
\usepackage{float}
\usepackage{verbatim}

\newtheorem{theorem}{Theorem}[section]
\newtheorem{lemma}[theorem]{Lemma}

\newtheorem{corollary}[theorem]{Corollary}
\newtheorem{construction}[theorem]{Construction}
\newtheorem{definition}[theorem]{Definition}
\newtheorem{remark}[theorem]{Remark}
\newtheorem{example}[theorem]{Example}

\begin{document}

\title{Cosets, voltages, and derived embeddings}

\author{Steven Schluchter\\ Department of Mathematical Sciences\\ George Mason University\\ Fairfax, VA 22030\\U.S.A.\\ \texttt{steven.schluchter@gmail.com}}

\date{}

\maketitle

\begin{abstract}An ordinary voltage graph embedding of a graph in a surface encodes a certain kind of highly symmetric covering space of that surface.  Given an ordinary voltage graph embedding of a graph $G$ in a surface with voltage group $A$ and a connected subgraph $H$ of $G$, we define special subgroups of $A$ that depend on $H$ and show how cosets of these groups in $A$ can be used to find topological information concerning the derived embedding without constructing the whole covering space.  Our strongest theorems treat the case that $H$ is a cycle and the fiber over $H$ is a disjoint union of cycles with annular neighborhoods, in which case we are able to deduce specific symmetry properties of the derived embeddings.  We give infinite families of examples that demonstrate the usefulness of our results.\end{abstract}

\section{Introduction}\label{section:introduction}

A cellular embedding of a graph in a surface encodes a cellular decomposition of the surface.  While a cellular graph embedding can be encoded in the form of a combinatorial object, say a rotation scheme or a walk double cover, the combinatorial objects themselves sometimes obscure the embeddings they encode: a rotation scheme encodes a graph embedding in the form of cyclic orderings of edges incident to a vertex, and a cycle or walk double cover encodes a graph embedding in the form of boundary walks of the faces of the embedding.  If $K$ is a connected 2-complex induced by a cellular graph embedding, then the combinatorial structure of the embedded graph is not necessarily clear if all one understands about $K$ is a set of lists of edges that form a walk double cover.  Similar difficulties arise from considering graph embeddings encoded in the form of rotation schemes.  Moreover, in the case of a walk double cover it is not immediately clear that $K$ is homeomorphic to a surface; there could be point singularities in $K$, which would mean that $K$ is homeomorphic to a pseudosurface (the result of a 2-manifold after a finite number of point identifications).

Another way to encode a cellular graph embedding is with an assignment of algebraic data to another cellular graph embedding.  This is the basic idea behind current graphs \cite[\S 4.4]{GT} and the various forms of voltage graph embeddings.  In each of these cases, the embedding that is encoded, called the \textit{derived embedding}, is a covering space of the surface (sometimes branched on the faces) containing the encoding, called the \textit{base embedding}.  Current graphs lend themselves easily to the encoding of triangular embeddings of complete graphs \cite[Example 4.4.1]{GT}.  Ordinary voltage graphs are equally as powerful a tool because they are related by duality to current graphs.  Ordinary voltage graphs and some related formulations have been used to construct specific kinds of embeddings of graphs in surfaces and pseudosurfaces.  Archdeadon in \cite{A} uses his medial-graph enhancement of ordinary voltage graph embeddings to construct orientable and nonorientable embeddings of specific complete bipartite graphs with specific bipartite topological dual graphs.  Archdeacon, Conder, and \v{S}ir\'{a}\v{n} in \cite{ACS} use specially constructed ordinary voltage graphs to construct infinitely many graph embeddings featuring multiple kinds of symmetries.  Ellingham and Schroeder in \cite{Ellingham} use ordinary voltage graph embeddings to construct an embedding of the complete tripartite graph $K_{n,n,n}$ in an orientable surface such that the boundary of each face is a Hamilton cycle.  More recently, in \cite{AS1}, ordinary voltage graph embeddings were used in the cataloging of all cellular automorphisms of all surfaces of Euler characteristic at least $-1$.

In these and other applications of (ordinary) voltage graph embeddings the desired embeddings are not well understood at a global level; it's difficult to conceive a totally transparent representation of the surface containing the derived embedding.  This is in part due to the fact that the encoded embedding has to be understood through two layers of encryption: a rotation scheme (or walk double cover) and the \textit{voltage assignment} of elements of a finite group, called the \textit{voltage group}, to the edges of the embedding.  The purpose of this article is to demonstrate that given an ordinary voltage graph embedding and a connected subgraph $H$ of $G$, there is basic topological information about the derived embedding that is contained in the cosets of specially constructed subgroups of $A$ that depend on $H$.  The strongest theorems treat the case in which $H$ is a connected 2-regular subgraph, which is called a \textit{cycle}.  

In Section \ref{section:definitions}, we review all necessary graph theory and topological graph theory including Archdeacon's extension of a voltage assignment to the subdivided medial graph of the base embedding.  In Section \ref{section:cosets}, we develop our theory for connected subgraphs $H$ of $G$ as an outgrowth of a theorem of Gross and Alpert.  Focusing specially on cycle subgraphs $C$ of $G$, we show that if each lift of $C$ has an annular neighborhood, then an understanding of the nature of the containments and intersections of cosets of the groups constructed in Section \ref{section:cosets} can lead to a basic understanding of the derived surface as a union of surfaces with boundary, whose boundary components are the lifts of $C$.  In Section \ref{section:examples}, we produce new examples of infinite families of derived embeddings that have specific properties.  The examples we produce are designed to highlight the kinds of insights obtained by using the results in Section \ref{section:cosets}.

\section{Definitions and Basic Information}\label{section:definitions}

\subsection{Graphs, graph embeddings, and 2-complexes}

For our purposes, a graph $G=(V,E)$ is a finite and connected multigraph.  An edge is a \emph{link} if it is not a loop.  A \emph{cycle} is a connected 2-regular subgraph of $G$.  A \emph{path} in $G$ is a sequence of vertices and edges $v_1e_1\ldots e_{n-1}v_n$ such that the vertices are all distinct and the edge $e_i$ connects vertices $v_i$ and $v_{i+1}$.  Given subgraphs $H_1$, $H_2$ of $G$, an \emph{$H_1$-$H_2$-path} is a path that has one of its end vertices in $H_1$, the other end vertex in $H_2$, and no other vertices in $H_1\cup H_2$; if $H_1=H_2$, then such a path is an $H_1$-path.  We let $D(G)$ denote the set of all darts (directed edges) in $G$.  Each dart $d\in D(G)$ has a \emph{head vertex} $h(d)$ and a \emph{tail vertex} $t(d)$.  We say that two darts on the same edge are \textit{opposites} of each other, and we adopt the convention that one dart is called the \textit{positive edge} and the other is called the \textit{negative edge}; $d^{-1}$ is the opposite of $d$.  A \textit{walk} $W$ is a sequence of darts $d_1d_2\ldots d_m$ such that $h(d_i)=t(d_{i+1})$.  If $h(d_m)=t(d_1)$, then we say that $W$ is a \textit{closed walk}.  We define an $H_1$-$H_2$-walk and an $H_1$-walk by analogy with an $H_1$-$H_2$-path and an $H_1$-path, respectively.  For an edge $e$ joining vertices $u$ and $v$, define the operation of \textit{subdividing the edge $e$} to be the operation of replacing the (directed) edge $e$ with a path $ue_1we_2v$ consisting of two edges.  We also define the \textit{subdivision of a graph} to be the operation of subdividing each edge of the graph.  Two graphs $G_1$ and $G_2$ are isomorphic if there is map $\phi\colon G_1\rightarrow G_2$ that bijectively maps vertices to vertices and edges to edges such that the incidence of edges at vertices is preserved. 

For the duration of this article, a \emph{surface} is a compact 2-manifold, $S$ shall denote a connected surface without boundary, and $\hat{S}$ shall denote a connected surface with boundary.  A \textit{cellular embedding} of $G$ in $S$ is an embedding that subdivides $S$ into 2-cells.  A \textit{proper embedding} of $G$ in $\hat{S}$ is an embedding that subdivides $\hat{S}$ into 2-cells; we require that the boundary components of $\hat{S}$ are contained in the image of $G$.  For the duration of this article, $G\rightarrow S$ and $G\rightarrow \hat{S}$ shall denote a cellular embedding and a proper embedding, respectively.  Two cellular or proper embeddings in the same surface are \emph{isomorphic} if there is a homeomorphism from the surface to itself that maps one embedding to the other.  Given $G\rightarrow S$, we say that that a cycle subgraph $C$ of $G$ is \textit{separating} if $S\setminus C$ is not connected and \textit{nonseparating} otherwise.  We will let $\left \{C_i(S),\partial\right \rbrace$ denote the $\mathbb{Z}_2$-chain complex induced by $G\rightarrow S$: $C_0(S),C_1(S),$ and $C_2(S)$ are the formal sums of vertices, edges, and faces of the 2-complex created by $G\rightarrow S$.  We let $Z(G)$ denote the subspace of $C_1(G)$ generated by the 1-chains inducing cycles in $G$, which is called the \textit{cycle space of $G$}.  We let $\partial\colon C_i(S)\rightarrow C_{i-1}(S)$ denote the usual boundary operator.  We use similar notation for the $2$-complex created by $G\rightarrow \hat{S}$.  For $X\in C_2(S)$, we let $S[X]$ denote the subcomplex of $\left \{C_i(S),\partial\right \rbrace$ consisting of the faces and all subfaces of $X$, and we let $G[X]$ denote the subgraph of $G$ consisting of all subfaces of $X$.  We can also treat $G$ as a 1-complex and adopt similar notation: for $X\in C_1(G)$, $X\in C_1(S)$ or $X\subset E(G)$, let $G[X]$ denote the subgraph of $G$ induced by $X$.  For $D_1\subset D(G)$, let $G[D_1]$ be the graph consisting of the edges whose darts appear in $D_1$.

For a subset $S_1$ of another set $U$, we let $S_1^c$ denote the complement of $S_1$ in $U$.  We will bend the notation somewhat, and let $S_1^c$ denote the $\mathbb{Z}_2$-sum of the elements of a $\mathbb{Z}_2$ vector space $U$ not appearing in $S_1$.  The proof of Lemma \ref{lemma:ComplexComplement} is straightforward, and therefore omitted.

\begin{lemma}\label{lemma:ComplexComplement} Given $G\rightarrow S$ and $X\in C_2(S)$, $\partial X= \partial X^c$.\end{lemma}

For a graph $G$ and a fixed $v$, $\mbox{star}(v)$ shall denote the set of all edges incident to $v$, links and loops; for $G\rightarrow S$, $U^*(v)$ denote an open set (in the usual Euclidean topology) in $S$ that contains $v$, intersects ends of edges of $\mbox{star}(v)$, and intersects no other edges or vertices of $G$.  We call $U^*(v)$ a \textit{vertex-star neighborhood of $v$}.  If one thickens $G$ such that the vertices become discs and the edges become rectangular strips glued to the discs, one produces what is called a \textit{band decomposition of $S$}: the \textit{$0$-bands} are the discs, the \textit{$1$-bands} are the rectangular strips, and the \textit{$2$-bands} are the discs glued to the $1$-bands and the $0$-bands.  Let $\rho: D(G) \rightarrow D(G)$ denote the permutation that takes $d$ to the next dart in the cyclic order of darts with tail vertex $t(d)$.  The order that follows this rotation is called the \textit{rotation on $t(d)$} and is denoted $t(d): d_1d_2\ldots$.  The permutation $\rho$ is called a \textit{rotation scheme on $G$}.  If one of the two possible orientations on any given 1-band joining $0$-bands is consistent with both of the orientations induced by $\rho$ on the joined $0$-bands (see \cite[Figures 3.13, 3.14]{GT} for enlightening diagrams), then we say that the edge corresponding to the 1-band is an \textit{orientation-preserving edge} and an \textit{orientation-reversing edge} otherwise.  We will call an orientation-preserving edge and an orientation-reversing edge a \emph{type-0 edge} and \emph{type-1 edge}, respectively.  Given $G\rightarrow S$, an \textit{orientation-reversing walk} is a walk consisting of an odd number of darts on type-1 edges, and an \textit{orientation-preserving walk} otherwise.  Similarly, a cycle is an \textit{orientation-reversing cycle} if it contains an odd number of orientation-reversing edges and an \textit{orientation-preserving cycle} otherwise.  Following the discussion in \cite[p.111]{GT} we may reverse the orientation on a $0$-band corresponding to a vertex $v$ without changing the embedding that corresponds to the rotation scheme: the rotation on $v$ is reversed (the rotation on $v$ is then given by $\rho^{-1}$) and the orientation type of each link incident to $v$ is switched.  Archdeacon in \cite{A} calls this process a \textit{local sign switch}.  Two graph embeddings are considered to be equivalent if their corresponding band decompositions differ by a sequence of local sign switches.

For a fixed vertex $v$ of $G$, consider $U^*(v)$, and let the \textit{corners} of $G\rightarrow S$ at $v$ refer to the components of $U^*(v)\setminus G[ \mbox{star}(v)]$.  The \textit{medial graph} is the graph whose vertices are the edges of $G$ and has an edge joining two vertices that correspond to edges of $G$ bounding a corner of $G\rightarrow S$.  The medial graph is clearly $4$-regular.  Each face of the medial graph embedding falls into one of two categories: those which correspond to the vertices of $G$ and those which correspond to the faces of $G\rightarrow S$.  An example of a medial graph of a cellular graph embedding appears in Figure \ref{fig:MedialExample}.

\begin{figure}[H]
\begin{center}
\includegraphics[scale=.5]{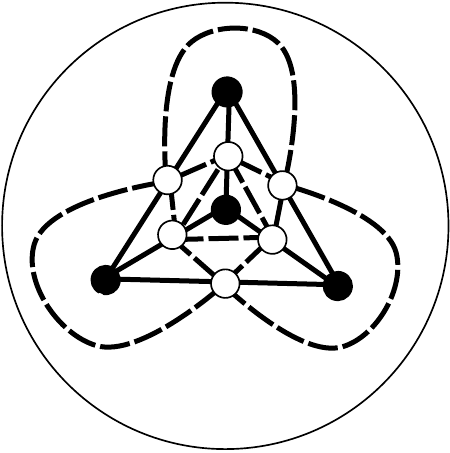}
\end{center}
\caption{A cellularly embedded graph $G$ in the sphere and the corresponding medial graph.  The medial graph has white vertices and dashed edges.}\label{fig:MedialExample}
\end{figure}

Throughout this article, for $G\rightarrow S$, $z\in C_1(S)$ shall denote a 1-chain of $S$ inducing a cycle in $G$.  We let $R(z)$ denote an open set in $S$ that contains only edges and vertices of $G[z]$, and intersecting only those edge ends that are incident to vertices of $G[z]$.  We will call $R(z)$ a \textit{ribbon neighborhood of $z$}; $R(z)$ is homeomorphic to an annulus or M\"obius band if $G[z]$ is an orientation-preserving cycle or orientation-reversing cycle, respectively

We will call a set of 1-chains inducing orientation-preserving cycles that have no vertices in common a \textit{set of 1-chains having property $\Delta$}.  Let $X=\left \{z_1,\ldots,z_m\right \rbrace$ denote a set of 1-chains having property $\Delta$.  We will call the connected sub 2-complexes of the 2-complex induced by $G\rightarrow S$ that are bounded by the cycles induced by the 1-chains of $X$ the \textit{z-regions of $S$ with respect to $X$}.  Define the \textit{z-graph of $S$ with respect to $X$}, which we denote $\Gamma(z_1,\ldots,z_m)$, to be the graph whose vertices are the z-regions of $S$ with respect to $X$ and whose edges are the 1-chains of $X$: an edge $z_j$ is incident to a vertex $v_\Gamma$ if $G[z_j]$ is contained in the boundary of $v_\Gamma$, and if both components of $R(z_j)\setminus G[z_j]$ are contained in $v_\Gamma$, then $z_j$ is a loop at $v_\Gamma$. 

\begin{remark}\label{remark:ZGraph} The purpose of introducing the z-graph is to introduce a combinatorial manner of capturing the incidence of z-regions at the cycles that bound them.  Since we are considering graphs embedded in surfaces, a cycle can bound up to two z-regions, and so it is fitting to describe the incidence of z-regions with a graph.  Given $G\rightarrow S$ and a set of 1-chains having property $\Delta$, it is easy to see that the corresponding z-graph is connected; for any two faces $f_1$, $f_2$ of $G\rightarrow S$, there is a sequence of faces $f_1f_{a_1}f_{a_2}\ldots f_2$ such that any two consecutive faces share at least one boundary edge.\end{remark}

\subsection{Ordinary voltage graphs and ordinary voltage graph embeddings}

Following \cite{GT}, we let $e$ denote the \textit{positive edge} on an edge $e\in E(G)$ and $e^{-}$ denote the \textit{negative edge} on $e$.  Let $A$ denote a finite group and let $1_A$ denote the identity element of $A$.  An \textit{ordinary voltage graph} is an ordered pair $\langle G, \alpha \rightarrow A \rangle$ such that $\alpha\colon D(G)\rightarrow A$ satisfies $\alpha(e^-)=\alpha(e)^{-1}$.  The group element $\alpha(e)$ is called the \textit{voltage} of $e$.  Associated to each ordinary voltage graph is a \textit{derived graph} $G^\alpha=(V\times A, E\times A)$.  The directed edge $(e,a)$ has tail vertex $(v,a)$ and head vertex $(v,a\alpha(e))$; it is a consequence of this and the conditions imposed on $\alpha$ that the dart $(e^{-},a\alpha(e))$ is the dart opposite $(e,a)$.  We will use the abbreviation $v^a$ for $(v,a)$ and $e^a$ for $(e,a)$.  We let $p\colon G^\alpha \rightarrow S$ denote the projection (covering) map satisfying $p(e^a)=e$ and $p(v^a)=v$.  For a walk $W=d_1d_2\ldots d_m$, let $\omega(W)=\alpha(d_1)\alpha(d_2)\ldots\alpha(d_m)$ denote the \textit{net voltage of $W$}.  If $W$ begins at a vertex $v$, we let $W_v^a$ denote the lift of $W$ beginning at $v^a$.  For a closed walk $W$, $W_v^a$ ends at the vertex at which $W_v^{a\omega(W)}$ begins, and we say that the lifts $W_v^a$ and $W_v^b$ are \textit{consecutive} if $a\omega(W)=b$ or $b\omega(W)=a$.  We call a set of lifts of the form \[ \left \{W_v^a,W_v^{a\omega(W)},\ W^{a\omega(W)^2},\ W^{a\omega(W)^3},\ \ldots,W^{a\omega(W)^{|\langle \omega(W)\rangle|-1}} \right \rbrace\] a \textit{a set of consecutive lifts of $W$}, and we let $\hat{W}_v^a$ denote the set of consecutive lifts of $W$ containing $W_v^a$.

Also described in \cite{GT}, an \textit{ordinary voltage graph embedding} of $G$ in $S$ is an ordered pair $\langle G \rightarrow S, \alpha \rightarrow A\rangle$, which is called a \textit{base embedding}.  Each base embedding encodes a \textit{derived embedding}, denoted $G^\alpha \rightarrow S^\alpha$, in the \textit{derived surface} $S^\alpha$.  We advise the reader that even though $S$ is assumed to be connected, $S^\alpha$ may not be connected.  Gross and Tucker in \cite{GT} describe the derived embedding according to rotation schemes, but we use Garman's manner of describing it; Garman points out in \cite{G} that since it is the lifts of facial boundaries that form facial boundaries in $S^\alpha$, $S^\alpha$ can be formed by ``identifying each component of a lifted region with sides of a 2-cell (unique to that component) and then performing the standard identification of edges from surface topology".  It is therefore permissible to have a base embedding in a surface with or without boundary; for each (directed) edge $e$ bounded on only one side by a face of $G\rightarrow \hat{S}$, each (directed) edge $e^a$ is bounded on only one side by a face of $G^\alpha \rightarrow \hat{S}^\alpha$.

For $\langle G\rightarrow S,\alpha \rightarrow A\rangle$ we let $S_v^a$ denote the component of $S^\alpha$ containing the vertex $v^a$, and for $\langle G,\alpha \rightarrow A\rangle$ we let $G_v^a$ denote the component of $G^\alpha$ containing $v^a$. We use similar notation for induced ordinary voltage graphs and ordinary voltage graph embeddings, e.g., for $I\in C_2(S)$ such that $S[I]$ is connected and $v\in V(S[I])$, $S[I]_v^a$ is the component of $S[I]^\alpha$ containing $v^a$.

The voltage group $A$ acts by left multiplication on $G^\alpha$.  For $c\in A$, let $c\cdot v^a=v^{ca}$, $c\cdot e^a=e^{ca}$.  This group action is  clearly regular (free and transitive) on the fibers over vertices and (directed) edges of $G^\alpha$, and so the components of $G^\alpha$ are isomorphic.  This action extends to a transitive (not necessarily free) action on the faces forming the fiber over a face of a base embedding, and so the components of $S^\alpha$ are homeomorphic as topological spaces and isomorphic as cellular complexes.  Per \cite[Theorem 4.3.5]{GT}, the graph covering map can be extended to a (branched) covering map of surfaces; it is not difficult to see that this covering map also extends to surfaces with boundary after considering Garman's construction of the derived embedding.  Moreover, for a walk $W$ based at a vertex $v$, if we let $\omega$ stand in place of $\omega(W)$ and $|\omega|$ stand in the place of $|\langle \omega(W)\rangle |$, we see here that the $A$-action extends to lifts of $W$ and sets of consecutive lifts of $W$: $c\cdot W_v^a=W_v^{ca}$ and \[ c\cdot \hat{W}_v^a = \left \{W_v^a,\ W_v^{ca\omega},\ W_v^{ca\omega^2},\ W_v^{ca\omega^3}, \ldots ,\ W^{ca\omega^{|\omega|-1}}\right \rbrace.\]

\subsubsection{Archdeacon's Medial Graph Enhancement}\label{subsubsection:ArchDeacon}		
Archdeacon in \cite{A} developed an extension of an ordinary voltage graph embedding to its subdivided medial graph and showed that the derived embedding of the subdivided medial graph is the subdivided medial graph of the derived embedding.  We state only the necessary definitions and theorems here.  For a vertex $v$ of $G$ to which no loops are incident and $c\in A$, define a \textit{local voltage modification at $v$} to be the result of replacing $\alpha(d)$ with $c\alpha(d)$ for all darts $d$ on edges in $\mbox{star}(v)$ with tail vertex $v$ (and $\alpha(d^{-1})$ with $\alpha(d^{-1})c^{-1}$).  For a (directed) edge $e$ such that $t(d)=u$, $h(d)=v$, and voltage $\alpha(e)$, define the operation of \textit{subdividing the voltage assignment to $e$} to be the operation of replacing the (directed) edge $e$ with a path $ue_1we_2v$ of length two such that $t(e_1)=u$, $h(e_1)=t(e_2)=w$, $h(e_2)=v$, and assigning the voltages $\alpha(e_1)=\alpha(e)$, $\alpha(e_1^{-})=\alpha(e)^{-1}$, $\alpha(e_2)=\alpha(e_2^{-})=1_A$; if we'd like to subdivide the voltage assignment to $e^-$ instead, we could perform a local voltage modification at $w$.  

\begin{lemma}\label{lemma:ArchdeaconModification}\cite[Lemma 3.2]{A} If two ordinary voltage assignments to the same graph or the same graph embedding differ by a local voltage modification, then they encode isomorphic derived graphs or derived embeddings, respectively.\end{lemma}

Given $\langle G,\alpha \rightarrow A\rangle$, the \textit{subdivided ordinary voltage graph} $\langle G',\alpha' \rightarrow A\rangle$ is formed by subdividing the voltage assignment to every directed edge of $G$. 

\begin{lemma}\label{lemma:ArchdeaconSubdivision}\cite[Lemma 3.3]{A} If $\langle G', \alpha' \rightarrow A\rangle$ is a subdivided ordinary voltage graph obtained from\\ $\langle G,\alpha \rightarrow A\rangle$, then $(G')^{\alpha'}$ is a subdivision of $G^\alpha$.\end{lemma}
We now describe how to transfer an ordinary voltage graph embedding to the (subdivided) medial graph.
\begin{construction}\label{construction:VoltTransferred}\cite[\S 4]{A}  Given $\langle G\rightarrow S,\alpha \rightarrow A\rangle$, the subdivided $G'$, and the corresponding subdivided voltage assignment, consider the medial graph $M$ of $G\rightarrow S$.  Subdivide every edge of $M$ to obtain the subdivided medial graph $M'$.  We choose a preferred direction $e^+$ (recall that this is denoted $e$) on a fixed edge $e$ of $G$ and let $v_e$ denote the corresponding vertex of the subdivided medial graph.  There are two corners of $G\rightarrow S$ at $t(e)$ and containing $e$ that correspond to two edges $e_1$, $e_2$ incident to $v_e$ in the subdivided medial graph.  Direct these two edges $e_1^+$ and $e_2^+$ so that they have head vertex $v_e$, and assign the voltage $\alpha(e^+)$ to both $e_1^+$ and $e_2^+$ and $\alpha(e^{-1})$ to $e_1^-$ and $e_2^-$.  The other darts of $M'$ whose head or tail vertex is $v_e$ are assigned voltage $1_A$.  We repeat this procedure for all other edges of $G$.  Since there are no loops in the subdivided $G'$ or the subdivided medial graph $M'$, we see that this \textit{transferred voltage assignment} to $M'$ is well defined, and it is consequence of Lemma \ref{lemma:ArchdeaconModification} that a choice of preferred direction can be reversed by a local voltage modification (see Figure \ref{fig:ArchdeaconMedial}).\hfill $\Box$ \end{construction}

\begin{figure}[H]
\begin{center}
\includegraphics[scale=.4]{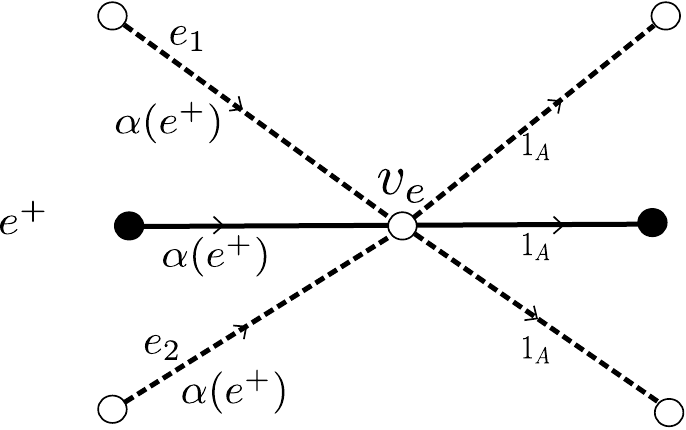}
\caption{The transferred medial voltage assignment.  The vertices and edges of $M'$ are white and dashed, respectively.}\label{fig:ArchdeaconMedial}
\end{center}
\end{figure}

Since the medial graph of an embedding combinatorially captures the incidence of faces and edges, Theorem \ref{theorem:DerivedMedial} is valuable for our purposes.

\begin{theorem}\label{theorem:DerivedMedial}\cite[Theorem 4.1]{A} Consider $\langle G\rightarrow S, \alpha \rightarrow A\rangle$ and let $M'$ be the subdivided medial graph with the transferred voltage assignment.  The derived graph $(M')^\alpha$ is the medial graph of the derived graph of the subdivided $G'$.\end{theorem}

Given $G\rightarrow S$, two consecutive darts $d_1d_2$ of a facial boundary walk, and a corner of $G\rightarrow S$ that is bounded by $G[\left \{d_1,d_2\right\rbrace]$, we define the operation of \textit{drawing an edge across a corner} to be the result of adding an edge joining $t(d_1)$ and $h(d_2)$ that intersects that corner.  Note that if $d_1$ and $d_2$ are opposites of each other, the edge drawn is a loop. Define the \textit{total graph} $T(G)$ of $G\rightarrow S$ to be the graph formed by subdividing each edge of $G$ and drawing in subdivided edges across each corner.  Thus, the total graph has both the medial graph and $G$ as minors.  We define an \textit{extended voltage assignment to $T(G)$} to be an assignment that is created by subdividing the ordinary voltage assignment of $G$, adding subdivided edges across each corner, and including the transferred voltage assignment from $G'$ to $M'$ described in Construction \ref{construction:VoltTransferred}.  We denote the extended voltage assignment $\alpha_E \rightarrow A$.  Note the existence of a function $\psi\colon T(G) \rightarrow T(G)$, which is the identity map on $G'$ and projects the darts (and edges) of $M'$ onto the darts of $G'$ bounding the corner of $G\rightarrow S$ across which the darts (and edges) are drawn, as in Figure \ref{fig:PsiMap}, where the images of vertices of $M'$ are further detailed.  Note that for each dart $d$ of $T(G)$, the extended voltage assignment $\alpha_E$ satisfies $\alpha(\psi(d))=\alpha_E(d)$.\medskip

\begin{figure}[H]
\begin{center}
\includegraphics[scale=.5]{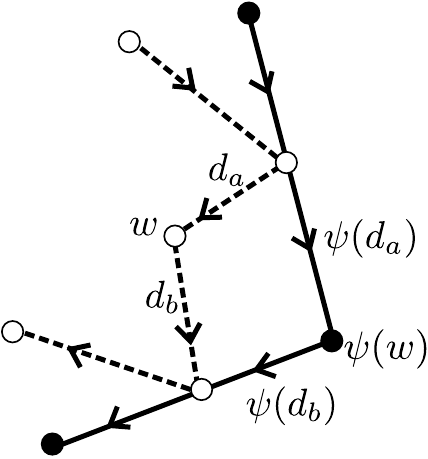}
\caption{An elucidation of the definition of the map $\psi\colon M'\rightarrow G'$ for one corner of $G\rightarrow S$.  The subdivided medial graph has white vertices and dashed edges.}\label{fig:PsiMap}
\end{center}
\end{figure} 

\section{Coset lemmas and coset theorems}\label{section:cosets} 

Given $\langle G\rightarrow S, \alpha \rightarrow A\rangle$, we examine specially constructed subgroups of $A$ corresponding to special induced subcomplexes of $\left \{C_i(S),\partial\right \rbrace$, and we establish relationships between certain cosets of these subgroups and induced subcomplexes of $\left \{C_i(S^\alpha),\partial\right \rbrace$.

\subsection{Cosets and constructions not requiring Archdeacon's medial-graph enhancement}\label{subsection:NonArchdeacon}

Following \cite[\S 2.5.1]{GT}, the net voltages of the closed walks in $G$ based at a vertex $v$ form a group, which is denoted $A(v)$.  We begin with a result of Gross and Alpert.

\begin{theorem}\label{theorem:GrossAlpert}\cite[Theorem 2.5.1]{GT} Given $\langle G,\alpha \rightarrow A\rangle$ and a fixed vertex $v$ of $G$, the vertices $v^a$ and $v^b$ are in the same component of $G^\alpha$ if and only if $a^{-1}b$ is an element of $A(v)$.\end{theorem}

Lemma \ref{lemma:ConnectedComponentsSalpha} establishes a relationship between left cosets of $A(v)$ in $A$ and the connected components of $S^\alpha$.  Lemma \ref{lemma:ConnectedComponentsSalpha} is a reformulation and strengthening of Theorem \ref{theorem:GrossAlpert}.

\begin{lemma}\label{lemma:ConnectedComponentsSalpha} Given $\langle G\rightarrow S, \alpha \rightarrow A\rangle$ and a fixed vertex $v$ of $G$, the map \[\phi_v \colon \left \{ aA(v):\ a\in A\right \rbrace \rightarrow \left \{ S_v^a: a\in A\right \rbrace \] defined by $\phi_v(aA(v))=S_v^a$ is a bijection.\end{lemma}

\begin{proof} Since $G_v^a$ is the 1-skeleton of $S_v^a$, it suffices to verify that $\phi_v$ is a bijection between the left cosets of $A(v)$ in $A$ and the connected components of $G^\alpha$.  Since $G$ is connected, for any vertex $u\in G$, there is a $u$-$v$-path $H$ that lifts to a total of $|A|$ distinct $p^{-1}(u)$-$p^{-1}(v)$-paths, each one joining a vertex in the fiber over $u$ to a distinct vertex in the fiber over $v$.  Thus, it suffices to show that $G_v^a=G_v^b$ if and only if $a^{-1}b\in A(v)$, which follows from Theorem \ref{theorem:GrossAlpert}.\end{proof}

Before we proceed any further, we need two technical lemmas.  The proof of Lemma \ref{lemma:CosetRedux} is a straightforward application of the free and transitive action of a group on itself by left multiplication, and therefore omitted.

\begin{lemma}\label{lemma:CosetRedux} Let $X$, $Y$, and $Z$ denote groups such that $X\le Y\le Z$.  For each $z\in Z$, the left coset $zY$ can be partitioned into left cosets $yX$ for $y\in zY$.\end{lemma}

 Lemma \ref{lemma:GrossAlpertRedux} establishes a relationship between the left cosets of a voltage group of net voltages of closed walks in a connected subcomplex $X$ and components of the fiber over $X$ contained in a specific component of $S^\alpha$.  The general formulation of Lemma \ref{lemma:GrossAlpertRedux} allows for easier applications to a wide range of cases.

\begin{lemma}\label{lemma:GrossAlpertRedux}  Let $S_B$ denote a surface, possibly with boundary, and let $G_B$ denote a connected graph properly embedded in $S_B$.  Let $A_B$ denote a finite group.  Fix $v\in C(G_B)$, and let $X$ denote a connected subcomplex of $S_B$ containing $v$.  Consider $\langle G_B\rightarrow S_B,\alpha \rightarrow A_B\rangle$.  Let $A_B(v)$ denote the local voltage group of closed walks based at $v$, let $A_B(v,X)$ denote the group of net voltages of closed walks in $X$ based at $v$.  For a fixed $a\in A_B$, the map \[\phi_X\colon \left \{bA_B(v,X):\ b\in aA_B(v)\right \rbrace \rightarrow \left \{X_v^b:\ b\in aA_B(v)\right \rbrace \] defined by \[\phi_X(bA_B(v,X))=X_v^b\] is a bijection.\end{lemma}

\begin{proof} We consider the case for which the domain of $\phi_X$ is the set of left cosets of $A_B(v,X)$ in $A_B(v)$, which is the case for which $a=1_A$.  By Lemma \ref{lemma:CosetRedux} and the transitive left action of $A$ on the fiber over $X$ this case will suffice.  Note that $G[X]_v^b$ is the 1-skeleton of $X_v^b$.  It suffices to show that $\phi_X$ is a bijection between the $bA_B(v,X)$ and the $G[X]_v^b$.  Since $G[X]$ is connected, for any other vertex $u$ of $G(X)$, there is a $u$-$v$-path in $G[X]$ that lifts to $|A_B|$ distinct $p^{-1}(u)$-$p^{-1}(v)$-paths in $G[X]^\alpha$, each one joining a vertex in the fiber over $u$ to a distinct vertex in the fiber over $v$.  Consider Theorem \ref{theorem:GrossAlpert}, and let $G[X]$, $A_B(v,X)$, and $A_B$ take the place of $G$, $A(v)$, and $A$, respectively.  It follows that $G[X]_v^b=G[X]_v^c$ if and only if $b^{-1}c$ is an element of $A_b(v,X)$.  The result follows.\end{proof}

Let $I \in C_2(S)$ be such that $S[I]$ is connected and assume that $v$ is a vertex of $G[I]$.  Let $A(v,S[I])$ denote the subgroup of $A(v)$ consisting of net voltages of closed walks in $G[I]$ based at $v$.  Note that \[A(v,S[I]) \le A(v) \le A.\]  By an application of Lemma \ref{lemma:CosetRedux} to the groups $A(v,S[I])$, $A(v)$, and $A$, we see that for each $a\in A$, the left coset $aA(v)$ can be partitioned into left cosets of the form $bA(v,S[I])$ for $b\in aA(v)$.  Lemma \ref{lemma:ConnectedComponentsB} establishes a relationship between left cosets of $A(v,S[I])$ contained in a left coset $aA(v)$ and components of $S[I]^\alpha$ contained in $S_v^a$.

\begin{lemma}\label{lemma:ConnectedComponentsB}  Given $\langle G\rightarrow S, \alpha \rightarrow A\rangle$ and a fixed $a\in A$.  The map \[ \phi_I \colon \left \{bA(v,S[I]):\ b\in aA(v)\right \rbrace \rightarrow \left \{S[I]_v^b:\ b\in aA(v)\right \rbrace\] defined by $\phi_I(bA(v,I)) = S[I]_v^b$ is a bijection.\end{lemma}

\begin{proof} This follows from Lemma \ref{lemma:GrossAlpertRedux} after letting $S$ take the place of $S_B$, $A$ take the place of $A_B$, $A(v)$ take the place of $A_B(v)$, $S[I]$ take the place of $X$, and $A(v,S[I])$ take the place of $A_B(v,X)$.\end{proof}

Let $y\in C_1(G[I])$ be such that $G[y]$ is connected and assume that $v$ is a vertex of $G[y]$.  Let $A(v,G[y])$ denote the subgroup of $A(v,S[I])$ consisting of net voltages of closed walks of $G[y]$ based at $v$.  Observe that \[A(v,G[y]) \le A(v,S[I])\le A(v).\]  By an application of Lemma \ref{lemma:CosetRedux} to the groups $A(v,G[y])$, $A(v,S[I])$, and $A(v)$, we see that for a fixed $a\in A$, each of the left cosets of the form $bA(v,S[I])$ contained in $aA(v)$ can be partitioned into left cosets $cA(v,G[y])$ for $c\in bA(v,S[I])$.  Lemma \ref{lemma:ConnectedComponentsC} establishes a relationship between the left cosets of $A(v,G[y])$ contained in $bA(v,S[I])$ and the components of $G[y]^\alpha$ contained in $S[I]_v^b$.  We omit the proof, which is straightforward in light of the proof of Lemma \ref{lemma:ConnectedComponentsB}.

\begin{lemma}\label{lemma:ConnectedComponentsC} Given $\langle G\rightarrow S, \alpha \rightarrow A\rangle$ with fixed $a\in A$, $b\in aA(v)$, the map \[\phi_y \colon \left \{ cA(v,G[y]):\ c \in bA(v,S[I]) \right \rbrace \rightarrow \left \{ G[y]_v^c:\ c \in bA(v,S[I])\right \rbrace\] defined by $\phi_y(cA(v,y)) = G[y]_v^c$ is a bijection.\end{lemma}

Let $W$ denote a closed walk in $G[y]$ based at $v$.  Clearly \[ \langle \omega(W)\rangle \le A(v,G[y])\le A(v,S[I])\le A(v)\le A.\]  By an application of Lemma \ref{lemma:CosetRedux} to the groups $\langle \omega(W) \rangle$, $A(v,G[y])$, and $A(v,S[I])$, we see that for a fixed $a\in A$ and $b\in aA(v)$, each of the left cosets $cA(v,G[y])$ contained in $bA(v,S[I])$ can be partitioned into left cosets of the form $d\langle \omega(W)\rangle$ for $d\in cA(v,G[y])$.  Lemma \ref{lemma:ConnectedComponentsD} establishes a relationship between the cosets of $\langle \omega(W)\rangle$  contained in $cA(v,G[y])$ and sets of consecutive lifts of $W$ in $G[y]_v^c$.

\begin{lemma}\label{lemma:ConnectedComponentsD}Fix $a\in A$, $b\in aA(v)$, and $c\in baA(v,S[I])$.  The map \[\phi_{\hat{W}}\colon \left \{ d\langle \omega(W)\rangle:\ d\in cA(v,G[y])\right \rbrace \rightarrow \left \{\hat{W}_v^d:\ d\in cA(v,G[y])\right \rbrace\] defined by $\phi_{\hat{W}}(\langle d\omega(W)\rangle)= \hat{W}_v^d$ is a bijection whose image is the set of consecutive lifts of $W$ contained in $G[y]_v^c$.\end{lemma}

\begin{proof} We treat the case for which  $c=1_A$.  By Lemma \ref{lemma:CosetRedux} and the transitive action of $A$ on the components of $G[y]^\alpha$, this case will suffice.	The map $\phi_{\hat{W}}$ is well defined and injective since $\hat{W}_v^{d_1} =\hat{W}_v^{d_2}$ if and only if $d_1=d_2x$ for some $x\in \langle \omega (W)\rangle$.  Surjectivity follows immediately from the definition of $\phi_{\hat{W}}$.

Note that $\hat{W}_v^{d_1}=\hat{W}_v^{d_2}$ are sets of consecutive lifts of $W$ in $G[y]_v^{1_A}$ if and only if $v^{d_1}$ and $v^{d_2}$ are contained in $G[y]_v^{1_A}$, which by letting $G[y]$ and $A(v,G[y])$ take the places of $G$ and $A(v)$ in Theorem \ref{theorem:GrossAlpert}, respectively, is true if and only if $d_1^{-1}d_2\in A(v,G[y])$.  The result follows.\end{proof}

The results of Theorem \ref{theorem:BigCosetTheorem} follow readily from Lagrange's theorem for cosets and the above lemmas, and the proofs are omitted. 

\begin{theorem}\label{theorem:BigCosetTheorem} Consider $\langle G\rightarrow S, \alpha \rightarrow A\rangle$, and let $I\in C_2(S)$ be such that $S[I]$ is connected.  Let $y\in C_1(S[I])$ be such that $G[y]$ is connected, and let $W$ denote a closed walk in $G[y]$ based at $v\in V(G[y])$.  
\begin{enumerate} 
\item{There are $\frac{|A|}{|A(v)|}$ components of $S^\alpha$.} \label{theorem:BigCosetTheoremPart1}
\item{There are $\frac{|A(v)|}{|A(v,S[I])|}$ components of $S[I]^\alpha$ contained in each component of $S^\alpha$.}\label{theorem:BigCosetTheoremPart2} 
\item{There are $\frac{|A(v,S[I])|}{|A(v,G[y])|}$ components of $G[y]^\alpha$ contained in each component of $S[I]^\alpha$.}\label{theorem:BigCosetTheoremPart3}
\item{There are $\frac{|A(v,G[y])|}{|\langle \omega(W)\rangle|}$ sets of consecutive lifts of $W$ in each component of $G[y]^\alpha$.}\label{theorem:BigCosetTheoremPart4}
\end{enumerate} \end{theorem}

Construction \ref{construction:SeparatingZGraph} allows us to algebraically describe a decomposition of any component $S_v^a$ of $S^\alpha$ in terms of the z-regions associated to the cycles forming the fiber over a separating cycle $G[z]$.  Since separating cycles are orientation-preserving cycles, each cycle in the fiber over $G[z]$ is also orientation preserving, and so the 1-chains forming the fiber over $z$ are a set of 1-chains with property $\Delta$.

\begin{construction}\label{construction:SeparatingZGraph} Let $z=\partial I$ for some $I\in C_2(S)$ be such that $G[z]$ is a cycle.  By Lemma \ref{lemma:ComplexComplement}, $z=\partial I^c$.  Let $W$ be an Eulerian walk of $G[z]$ based at $v\in V(G[z])$.   Let $c_1,\ \ldots\, c_j,\ j= \frac{|A(v)|}{|\langle \omega(W)\rangle|}$ denote representatives of the left cosets of $\langle \omega(W)\rangle$ contained in $aA(v)$.  By Lemmas \ref{lemma:ConnectedComponentsB} and \ref{lemma:ConnectedComponentsC}, we may conclude that for the 1-chains $z_v^{c_1},\ldots,z_v^{c_j}$ contained in $S_v^a$, the vertices of the corresponding z-graph $\Gamma(z_v^{c_1},\ldots,z_v^{c_j})$ correspond bijectively to the left cosets contained in the union 
\[ \left \{bA(v,S[I]):\ b\in aA(v)\right \rbrace \bigcup \left \{bA(v,S[I^c]):\ b\in aA(v)\right \rbrace.\]   Similarly, the left cosets of $\langle \omega(W)\rangle$ contained in $aA(v)$ correspond bijectively to the 1-chains $z_v^c$ contained in $S_v^a$, which are the edges of $\Gamma(z_v^{c_1},\ldots,z_v^{c_j})$.  Thus, implicitly using the aforementioned bijections, we may say that the edges of $\Gamma(z_v^{c_1},\ldots,z_v^{c_j})$ are incident to the vertices of $\Gamma(z_v^{c_1},\ldots,z_v^{c_j})$ that contain them.\end{construction}

\begin{remark}\label{remark:BipartiteCosetGraph} The graph  $\Gamma(z_v^{c_1},\ldots,z_v^{c_j})$ constructed in Construction \ref{construction:SeparatingZGraph} is bipartite.  The bipartition of the vertex set is into two sets $S_1$ and $S_2$, whose elements correspond to the elements of $\left \{bA(v,S[I]):\ b\in a(v)\right \rbrace$ and $\left \{bA(v,S[I^c]):\ b\in aA(v)\right \rbrace$, respectively.\end{remark}

\subsection{Cosets requiring Archdeacon's medial-graph enhancement}\label{subsection:ArchdeaconCosets}

Recall from Section \ref{subsubsection:ArchDeacon} the subdivided medial graph $M'$, the total graph $T(G)$ associated to a cellular graph embedding, and the extended voltage assignment to $T(G)$.  Since $(M')^\alpha$ captures the incidence of faces and edges of $G^\alpha\rightarrow S^\alpha$, we introduce Definition \ref{definition:SpecialClaw} as a bookkeeping tool for computation that will allow for easier proofs of results and constructions of more z-graphs.

\begin{definition}\label{definition:SpecialClaw} Consider $\langle G\rightarrow S, \alpha \rightarrow A\rangle$, two adjacent vertices $u$ and $v$ of $G$ that are joined by an edge $e$, and a choice of preferred direction of $e$.  Let $v_e$ denote the vertex of the subdivided medial graph corresponding to $e$, and consider Figure \ref{figure:SpecialClaw}.  The subgraph $T(G)[\left \{(v,v_e),(w,v_e),(v_e,y)\right \rbrace]$ is called the \textit{special claw corresponding to $e$}.\end{definition}

\begin{figure}[H]
\begin{center}
\includegraphics[scale=.55]{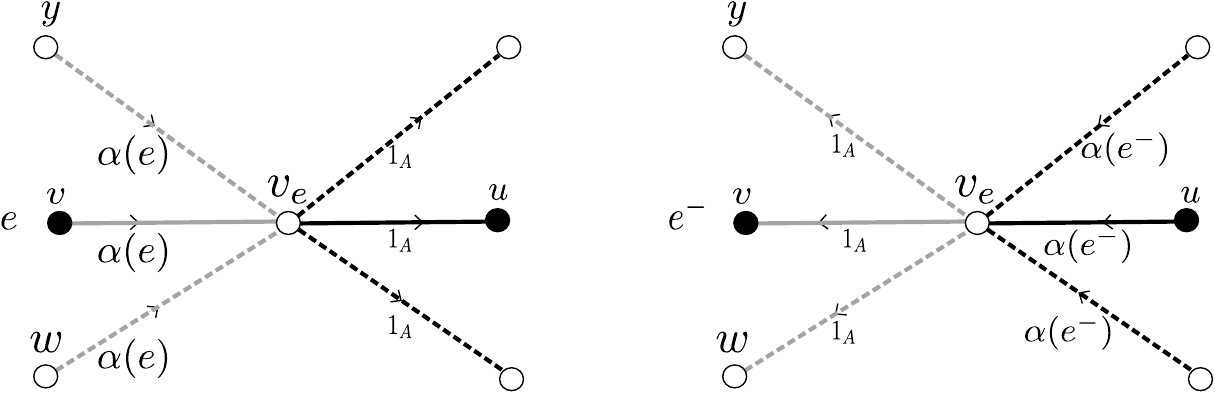}
\caption{The edges of the special claw corresponding to $e$ appear in gray on the left hand side of both graphics, each corresponding to a choice of preferred direction on $e$.  The subdivided medial graph $M'$ has white vertices and dashed edges.}\label{figure:SpecialClaw}
\end{center}
\end{figure}

\begin{remark}\label{remark:VoltageClawSuperscripts}No matter the choice of preferred direction of $e$, any walk joining any two of the vertices $w$, $v$, and $y$ in the special claw corresponding to $e$ has net voltage $1_A$.  Therefore, the vertices in each component in the fiber over the special claw corresponding to $e$ will have vertices $v^a$, $w^a$, $y^a$ with identical group element superscripts.\end{remark}

Let $A'(w')$ denote the group of net voltages of closed walks in $M'$ based at $w'\in V(M')$.  Lemma \ref{lemma:MedialGroupsEqual} is immediate in light of the projection map $\phi: M'\rightarrow G'$.

\begin{lemma}\label{lemma:MedialGroupsEqual}For a fixed vertex $v$ of $G$, an edge $e$ incident to $v$, the special claw corresponding to $e$, and a vertex $w'$ of $M'$ that takes the place of either of the vertices $w$ or $y$ in Definition \ref{definition:SpecialClaw}, $A'(w')=A(v)$.\end{lemma}

Consider $G\rightarrow S$ and $y\in C_1(G)$ such that $G[y]$ is connected.  We introduce Definition \ref{definition:TransverseCrossing} for the sake of presenting it in the most general context, and alert the reader to the fact that it will be used in distinct settings in Sections \ref{subsubsection:OrientationReversingCycle} and \ref{subsubsection:OrientationPreservingCycle}.

\begin{definition}\label{definition:TransverseCrossing} Let $W'$ be a walk in $M'$.  We say that $W'$ transversely crosses $G[y]$ if  $W'$ contains a subwalk $d_id_{i+1}$ such that $d_i$ intersects one component of $U^*(h(d_i))\setminus G[y]$ and $d_{i+1}$ intersects the other.\end{definition}

Fix $w'\in V(M')$.  Let $A^\veebar(w',G[y])$ denote the subgroup of $A'(w')$ of walks in $M'$ that do not transversely cross $G[y]$.  For a vertex $w'\in V(M')$ and $z\in Z(G)$ inducing a cycle such that the cycles of $G[z]^\alpha$ are orientation-preserving cycles, we let $\bar{S}_{w'}^b(z)$ denote the z-region of $S^\alpha$ with respect to the fiber over $z$ that contains $w^{'b}$.

\subsubsection{Cycles over orientation-reversing cycles}\label{subsubsection:OrientationReversingCycle}

Consider $\langle G\rightarrow S, \alpha \rightarrow A\rangle$ for a nonorientable $S$ and $z\in Z(G)$ such that $G[z]$ is an orientation-reversing cycle.  Let $W=d_1d_2\ldots d_k$ denote an Eulerian walk of $G[z]$ based at $v\in V(G[z])$.  For the remainder of Section \ref{subsubsection:OrientationReversingCycle}, we let $\omega$ denote $\omega(W)$ and $|\omega|$ denote $|\langle \omega (W)\rangle |$.  We also assume that $|\omega|$ is even so that the cycles forming the fiber over $G[z]$ are orientation preserving, per Lemma \ref{lemma:OrientationReversingVoltage}.

\begin{lemma}\label{lemma:OrientationReversingVoltage} Each cycle of $G[z]^\alpha$ is orientation preserving if and only if $|\omega|$ is even.\end{lemma}

\begin{proof} A cycle is orientation reversing if and only if it has an odd number of orientation-reversing edges.  Following Part \ref{theorem:BigCosetTheoremPart4} of Theorem \ref{theorem:BigCosetTheorem}, there are $|\omega|$ consecutive lifts of $W$ that single cover the edges of each component of $G[z]^\alpha$.  The result follows.\end{proof}

Consider a vertex-star neighborhood $U^*(v)$ of $v$, and consider an orientation of $U^*(v)$ induced by the rotation $\rho$ on $v$.  Consider also the two components of $U^*(v)\setminus G[z]$.  It is thus reasonable to denote one component of $U^*(v)\setminus G[z]$ the \textit{clockwise side of $v$} and the other open set the \textit{counterclockwise side of $v$} according to the orientation on $U^*(v)$ induced by $\rho$ and the direction of the darts $d_1$ and $d_k$, as in Figure \ref{figure:EastWestSides}.

\begin{figure}[H]
\begin{center}
\includegraphics[scale=.45]{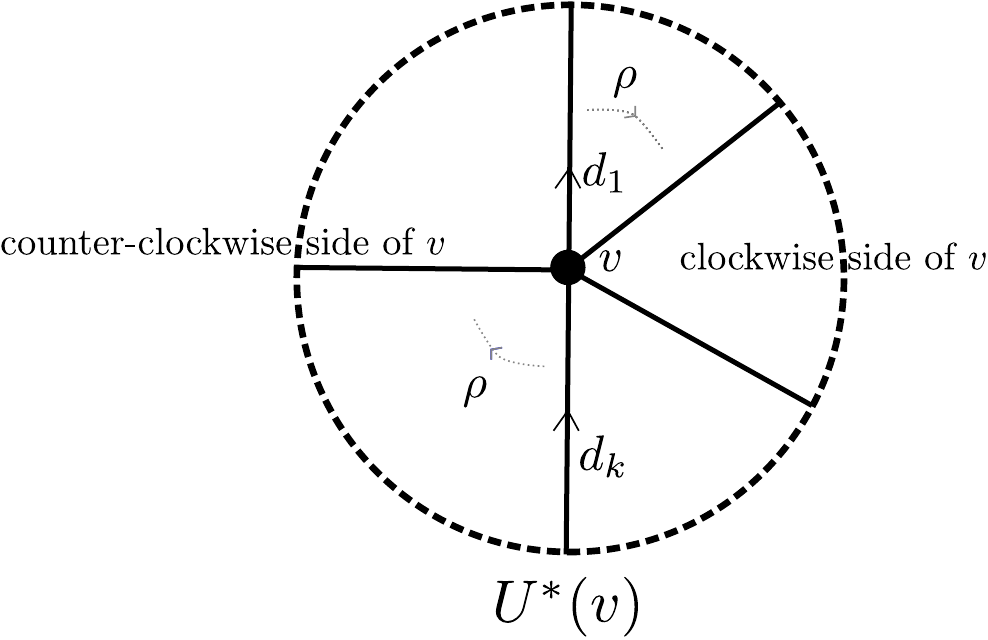}
\caption{The clockwise and counterclockwise sides of $v$.}\label{figure:EastWestSides}
\end{center}
\end{figure}

Recall the covering map $p\colon S^\alpha \rightarrow S$ and let $U^a(v)$ denote the component of $p^{-1}(U^*(v))$ containing $v^a$.  We lift the distinction of clockwise and counterclockwise sets to the elements of $\left \{U^a:\ a\in A\right \rbrace$; let the \textit{clockwise side of $v^a$} refer to the component of $U^a(v)\setminus G[z]_v^a$ that $p$ maps to the clockwise side of $v$ and the \textit{counterclockwise side of $v^a$} refer to the other component.

We now consider an extended voltage assignment to the total graph $T(G)$ associated to $G\rightarrow S$.  For clarity's sake we require that $U^*(v)$ contains the vertices $v,\ x,\ w,$ and $y$ of the special claw corresponding to $E(d_1)$.  Assume without loss of generality that the clockwise side of $v$ contains $w$ and the counterclockwise side of $v$ contains $y$.

Consider a ribbon neighborhood $R(z_v^{1_A})$ of $G[z]_v^{1_A}$ containing $U^{1_A}(v)$.  Let the \textit{east side} of $z_v^{1_A}$ refer to the component of $R(z_v^{1_A})\setminus G[z]_v^{1_A}$ containing the clockwise side of $v^{1_A}$, and the \textit{west side} of $z_v^{1_A}$ refer to the other component; the east side of $z_v^{1_A}$ contains $w^{1_A}$ and the west side  contains $y^{1_A}$.  Since each lift of $W$ single covers an odd number of orientation-reversing edges of $G^\alpha$, we see that the east side of $z_v^{1_A}$ contains the counterclockwise side of $v^\omega$.  Continuing in this way, we see that for $l$ odd, the east side of $z_v^{1_A}$ contains the counterclockwise side of $v^{\omega^l}\!$, and the west side of $z_v^{1_A}$ contains the clockwise side of $v^{\omega^l}$.  Moreover, for $l$ even, we see that the east side of $z_v^{1_A}$ contains the clockwise side of $v^{\omega^l}$ and the counterclockwise side of $v^{\omega^l}$.  Figure \ref{figure:OneCycleTrivialExample} illuminates this discussion.  Observe the alternation of the rotations of the vertices in the fiber over $v$ following each lift of $W$.  For the sake of a more intelligible picture, we show neither the total graph embedding nor the derived total graph embedding.

\begin{figure}[H]
\begin{center}
\includegraphics[scale=.5]{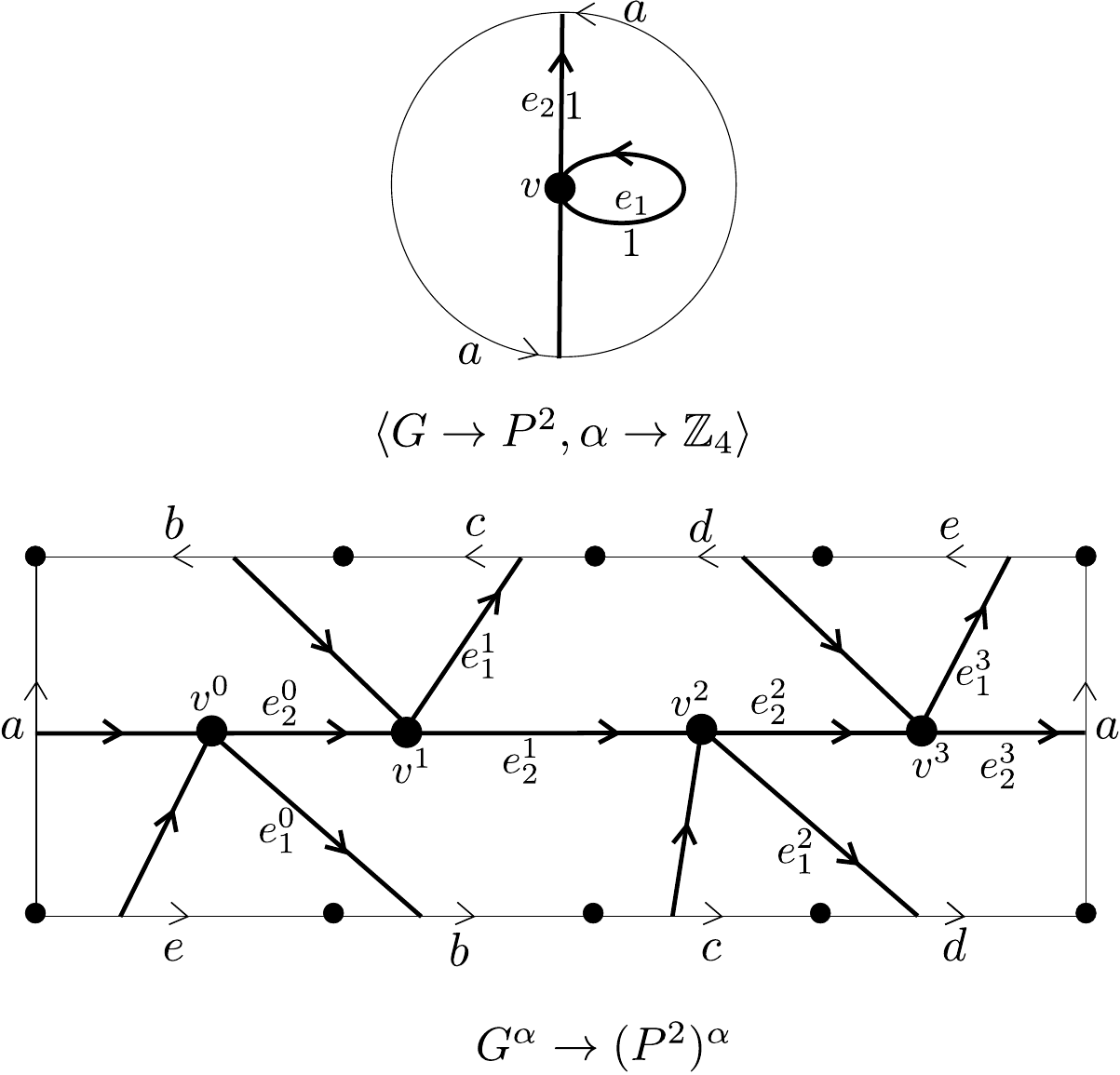}
\caption{A nonorientable ordinary voltage graph embedding in the projective plane $P^2$ and its derived embedding.  Local sign switches have been applied to the vertices $v^1$ and $v^3$.} \label{figure:OneCycleTrivialExample} 
\end{center}
\end{figure}

\begin{theorem}\label{theorem:BigOrientationReversingCycleTheorem}Consider $\langle G\rightarrow S,\alpha \rightarrow A\rangle$ for a nonorientable $S$ and let $z\in Z(G)$ induce an orientation-reversing cycle such that the Eulerian walk $W$ of $G[z]$ has the property that $|\omega|$ is an even integer.\begin{enumerate} 
\item{Each z-region of $S^\alpha$ with respect to $\left \{ z_v^a:\ a\in A\right \rbrace$ contains the same nonzero number of vertices in the fiber over $w$.}\label{theorem:BigOrientationReversingCycleTheoremPart1}

\item{For each $a\in A$, the component $S_v^a$ of $S^\alpha$ contains either one or two z-regions with respect to the fiber over $z$.}\label{theorem:BigOrientationReversingCycleTheoremPart2} 

\item{If $\langle \omega \rangle \neq A(v)$, then each component of $G[z]^\alpha$ is nonseparating.}\label{theorem:BigOrientationReversingCycleTheoremPart3}\end{enumerate}\end{theorem}

\proof By the discussion preceding the statement of Theorem \ref{theorem:BigOrientationReversingCycleTheorem}, we may conclude that both the east and west sides of $z_v^{1_a}$ contain the same nonzero number of lifts of the clockwise and counterclockwise sides of $v$.  Since $A$ is transitive on the fiber over $G[z]$, it follows that each of the components of $R(z_v^a)\setminus G[z]_v^a$ intersects the same number of components of the lifts the clockwise and counterclockwise sides of $v$.  Part \ref{theorem:BigOrientationReversingCycleTheoremPart1} follows.

To prove Part \ref{theorem:BigOrientationReversingCycleTheoremPart2}, first note that the action of $\omega$ on $S_v^a$ can be seen to be a glide reflection on the components of $G[z]^\alpha$.  If there exist two z-regions of $S^\alpha$ relative to the 1-chains in the fiber over $z$ that share a common boundary component, then the action of $\omega$ on $S^\alpha$ switches those z-regions, and so we see that any two of these z-regions are homeomorphic.  If there are two or more components of $G[z]^\alpha$ bounding a z-region $F$, then $F$ cannot be bounded by two or more distinct z-regions since $F$ cannot be switched with two or more distinct z-regions simultaneously.  Part \ref{theorem:BigOrientationReversingCycleTheoremPart2} follows.

We now move to proving Part \ref{theorem:BigOrientationReversingCycleTheoremPart3}.  Let $I = \sum_{f\in C_2(S)} f$.  Since $G[z]$ is a cycle, it follows that $\langle \omega \rangle= A(v,G[z])$, and so, by Lemma \ref{lemma:ConnectedComponentsC}, we have that $\left \{z_v^a:\ a\in A(v)\right \rbrace$ is the set of 1-chains inducing the cycles in $G[z]^\alpha$ that are contained in $S_v^{1_A}$.  If there is only one z-region of $S_v^{1_A}$ with respect to $\left \{z_v^a:\ a\in A(v)\right \rbrace$, then it is apparent that no single $G[z]_v^a$ is separating.  Now, assume that there are two z-regions of $S_v^{1_A}$ with respect to $\left \{z_v^a:\ a\in A(v)\right \rbrace$.  If $A(v)\neq \langle \omega \rangle$, then $\frac{|A(v,S[I])|}{| \omega |} >1$ and so, by Part \ref{theorem:BigCosetTheoremPart3} of Theorem \ref{theorem:BigCosetTheorem}, there is more than one component of $G[z]^\alpha$ bounding both z-regions of $S_v^{1_A}$ with respect to $\left \{z_v^a:\ a\in A(v)\right \rbrace$.  Part \ref{theorem:BigOrientationReversingCycleTheoremPart3} follows by transitivity of the action of $A$ on the components of $S^\alpha$.\hfill $\blacksquare$

\subsubsection{Cycles over nonseparating orientation-preserving cycles}\label{subsubsection:OrientationPreservingCycle}

Here $S$ can be an orientable or nonorientable surface.  Consider $\langle G\rightarrow S,\alpha \rightarrow A\rangle$, and let $z\in Z(G)$ induce a nonseparating orientation-preserving cycle.  Let $W$, $\omega$, $|\omega|$ and $v$ denote the same things they did in Section \ref{subsubsection:OrientationReversingCycle}, but here, $|\omega|$ is allowed to be any positive integer.  Consider the extended voltage assignment to the total graph $T(G)$, the special claw corresponding to $E(d_1)$, and let $w'$ and $y'$ take the place of the vertices $w$ and $y$ appearing in Definition \ref{definition:SpecialClaw}.

\begin{theorem}\label{theorem:OrientationPreservingNonSeparating}  Fix $a\in A$ and consider $S_v^a$.  \begin{enumerate}\item{Each of the z-regions of $S_v^a$ with respect to $\left \{z_v^b:\ b\in aA(v)\right \rbrace$ contains a vertex in the fiber over $w'$ and a vertex in the fiber over $y'$.}\label{theorem:OrientationPreservingNonSeparatingPartOne}\item{If there is more than one z-region of $S_v^a$ with respect to $\left \{z_v^b:\ b\in aA(v)\right \rbrace$, then each z-region is bounded by an even number of cycles of $G[z]^\alpha$.}\label{theorem:OrientationPreservingNonSeparatingPartTwo}\end{enumerate}\end{theorem}

\begin{proof} By the action of $A$ on itself by left multiplication and Lemma \ref{lemma:CosetRedux}, it suffices to consider the case in which $a=1_A$.  By Lemma \ref{lemma:ConnectedComponentsC}, the 1-chains contained in $\left \{z_v^b:\ b\in A(v) \right \rbrace$ are the 1-chains over $z$ that are contained in $S_v^{1_A}$.  If there is only one z-region of $S_v^{1_A}$ with respect to the $z_v^b$, then Part \ref{theorem:OrientationPreservingNonSeparatingPartOne} follows.

We now treat the case in which there are multiple z-regions of $S_v^{1_A}$ with respect to the $z_v^b$.  Since $G[z]$ is nonseparating, there is a $w'$-$y'$-walk $W'$ in $M'$ that does not transversely cross $G[z]$.  Let $\omega'$ denote $\omega(W')$.  Therefore, for each $w^{'c}$, there is a unique $W^{'c}$ that begins at $w^{'c}$, ends at $y^{'c\omega'}$, and does not transversely cross any cycle in $\left \{G[z]_v^b:\ b\in A(v)\right \rbrace$.  Fix $b\in A(v)$.  By Remark \ref{remark:VoltageClawSuperscripts},  one component of $R(z_v^b)\setminus G[z]_v^b$ contains $w^{'b}$ and other contains $y^{'b}$.  Following a lift of $W'$, we see that the z-region that contains $w^{'b}$ also contains $y^{'b\omega'}$.  Following a lift of $W'$ backwards, we see that the z-region that contains $y^{'b}$ also contains $w^{'b\omega^{-1}}$.  Part \ref{theorem:OrientationPreservingNonSeparatingPartOne} follows. 

We now move to proving Part \ref{theorem:OrientationPreservingNonSeparatingPartTwo}.  For $b\in aA(v)$, let the \textit{east side of $z_v^{b}$} denote the component of $p^{-1}(R(z)\setminus G[z])$ containing $w^{'b}$ and let the \textit{west side of $z_v^{b}$} denote the component of $p^{-1}(R(z)\setminus G[z])$ containing $y^{'b}$.  Since each lift of $W$ is orientation preserving, we see that the east side of $z_v^b$ contains each of the vertices \[w^{'b},\ w^{'b\omega},\ldots, w^{'b\omega^{|\omega|-1}},\] and the west side contains each of the vertices  \[y^{'b},\ y^{'b\omega},\ldots, y^{'b\omega^{|\omega|-1}}.\]  For a fixed $b$, consider a z-region $F$ with respect to $\left \{z_v^b:\ b\in aA(v)\right \rbrace$, and assume for the sake of contradiction that $F$ is bounded by an odd number of cycles of $G[z]^\alpha$.  Since $F$ is bounded by an odd number of cycles of $G[z]^\alpha$, $F$ must contain a number of east sides of and west sides of the $z_v^b$ of different parity.  Since $F$ contains at least one vertex in each the fibers over $w'$ and $y'$, $F$ contains at least one east side and one west side of the $z_v^b$ in its boundary.  Without loss of generality, assume that $F$ contains a nonzero odd number $q$ of east sides and an nonzero even number $r$ of west sizes of the $z_v^b$.  Using the action of $A$ on $S_v^{1_A}$ and the associated symmetry, each z-region $\bar{S}_{w'}^d$ with respect to the $z_v^b$ can be expressed as $c\cdot F$ for some $c\in A(v)$.  Assuming that there are a total of $n$ z-regions we see that $S_v^{1_A}$ contains $qn$ east sides and $rn$ west sides, which contradicts that there must be the same number of east and west sides in $S_v^{1_A}$ bounded by the $G[z]_v^b$.  This completes Part \ref{theorem:OrientationPreservingNonSeparatingPartTwo}.\end{proof}

\begin{corollary}\label{corollary:OrirntationPreservingNonSeparating} Each of the cycles in $\left \{G[z]_v^b:\ b\in aA(v)\right \rbrace$ is nonseparating.\end{corollary}
\begin{proof} This is immediate if there is only one z-region of $S_v^a$ with respect to $\left \{z_v^b:\ b\in aA(v)\right \rbrace$.  If there is more than one such z-region, then Part \ref{theorem:OrientationPreservingNonSeparatingPartTwo} of Theorem \ref{theorem:OrientationPreservingNonSeparating} implies the conclusion since the corresponding z-graph is a regular graph with each vertex of nonzero even degree.\end{proof}

\subsubsection{Cosets and z-graphs from Archdeacon's enhancements}

Constructions \ref{construction:OrientationReversingZGraph} and \ref{construction:OrientationPreservingZGraph} allow us to algebraically describe a decomposition of any component $S_v^a$ of $S^\alpha$ in terms of the z-regions associated to the cycles forming the fiber over an orientation-reversing cycle that lifts to orientation-preserving cycles, or to the cycles forming the fiber over orientation-preserving cycles, respectively.

Consider $\langle G\rightarrow S,\alpha \rightarrow A\rangle$ and $z\in Z(G)$ such that $G[z]$ induces a nonseparating cycle.  Fix $v\in V(G)$ and let $W=d_1d_2\ldots d_k$ be an Eulerian walk of $G[z]$ based at $v\in V(G[z])$. Consider the extended voltage assignment to the total graph $T(G)$, the special claw corresponding to $E(d_1)$ and let $w'$ and $y'$ take the place of $w$ and $y$ in Definition \ref{definition:SpecialClaw}, respectively.  Implicitly using the equality $A'(w')=A(v)$ occurring in Lemma \ref{lemma:MedialGroupsEqual} we are able to establish a bijection between left cosets of $A'(w',G[z])$ contained in $aA(v)$ and the z-regions of $S_v^a$ with respect to $\left \{z_v^c:\ c\in aA(v)\right \rbrace$, which, by Lemma \ref{lemma:ConnectedComponentsC}, is the set of 1-chains in the fiber over $z$ that is contained in $S_v^a$.

\begin{lemma}\label{lemma:ZRegionsCosets}For a fixed $a\in A(v)$, the map $\phi_{\veebar}\colon \left \{bA^\veebar(w',G[z]):\ b\in aA'(w')\right \rbrace \rightarrow \left \{\bar{S}_{w'}^b: b\in aA'(w')\right \rbrace$ defined by \[ \phi_{\veebar}(bA^\veebar(w',G[z])) = \bar{S}_{w'}^b(z)\] is a bijection.\end{lemma}

\begin{proof}  By Part \ref{theorem:BigOrientationReversingCycleTheoremPart1} of Theorem \ref{theorem:BigOrientationReversingCycleTheorem} and Part \ref{theorem:OrientationPreservingNonSeparatingPartOne} of Theorem \ref{theorem:OrientationPreservingNonSeparating}, each $\bar{S}_{w'}^b\ b\in A'(w')$ contains at least one vertex in the fiber over $w'$.  By Lemma \ref{lemma:GrossAlpertRedux} and the $A$-action on $S^\alpha$, it suffices to consider the case $a=1_A$.  It is apparent that $\bar{S}_{w'}^{d_1}=\bar{S}_{w'}^{d_2}$ if and only if there is a $w^{'d_1}$-$w^{'d_2}$-walk $W'$ in $M'^\alpha$ that does not transversely cross any component of $G[z]^\alpha$.  Given such a $w^{'d_1}$-$w^{'d_2}$-walk $W'$, it follows that $p(W^{'})$ is a $w'$-walk in $M'$ of net voltage $d_1^{-1}d_2$ that does not transversely cross $G[z]$.  The result follows.\end{proof}

After using the equality $A'(w')=A(v)$ of Lemma \ref{lemma:MedialGroupsEqual}, Theorem \ref{theorem:BigCosetTheorem2} follows from Lemma \ref{lemma:ZRegionsCosets} just as Theorem \ref{theorem:BigCosetTheorem} followed from Lemmas \ref{lemma:ConnectedComponentsSalpha}, \ref{lemma:ConnectedComponentsB}, \ref{lemma:ConnectedComponentsC}, and \ref{lemma:ConnectedComponentsD}.  The proof is omitted.

\begin{theorem}\label{theorem:BigCosetTheorem2} Let $z\in Z(G)$ induce a nonseparating cycle such that $G[z]^\alpha$ is a disjoint union of orientation-preserving cycles.  There are \[\frac{|A(v)|}{|A'(w',G[z])|}\] z-regions of $S_v^a$ with respect to $\left \{z_v^b:\ b\in aA(v)\right \rbrace$.\end{theorem}

For $a\in A$, let $b_1, b_2,\ldots,b_j$ denote representatives of the left cosets of $\langle \omega(W)\rangle $ in $aA(v)$.  Also, for the purposes  of Constructions \ref{construction:OrientationReversingZGraph} and \ref{construction:OrientationPreservingZGraph}, we let $\omega$ denote $\omega(W)$ and $|\omega|$ denote $|\langle \omega\rangle|$.

\begin{construction}\label{construction:OrientationReversingZGraph} Let $z$ induce an orientation-reversing cycle.  Let $|\omega|$ be even so that, per Lemma \ref{lemma:OrientationReversingVoltage}, the cycles forming the fiber over $G[z]$ are orientation-preserving.  Since the east and west sides of $z_v^{1_A}$ contain alternating vertices in the list \[w'^{1_A},w^{'\omega},w^{'\omega^2},w^{'\omega^3},\ldots,w^{'\omega^{|\omega|-1}},\] we see that the z-regions containing $G[z]_v^{1_A}$ in their boundaries contain all of the vertices in one (or both) of the lists \[w^{'1_A},\ w^{'\omega^2},\ldots, w^{'\omega^{|\omega|-2}}, \]  \[w^{'\omega},w^{'\omega^3},\ldots,w^{'\omega^{|\omega|-1}}.\]  In the case of there being only one z-region of $S_v^{1_A}$ with respect to $\left \{z_v^a:\ a\in A(v)\right \rbrace$, then that z-region will contain all of the vertices of \[w^{'1_A}, w^{'\omega},w^{'\omega^2},w^{'\omega^3},\ldots, w^{'|\omega|-1}.\]  Since the $A$-action on $S^\alpha$ is free on the vertices in the fiber over $w'$, we see that, for each $c\in A$, each component of $R(z_v^{c})\setminus G[z]_v^c$ contains the vertices in the fiber over $w'$ whose superscripts  are those found in one of the left cosets \[\left \{c, c{\omega^2},\ c{\omega^4}\ldots,\ c{\omega^{|\omega|-2}} \right \rbrace,\] \[\left \{c\omega,\ c{\omega^3},\ c{\omega^5},\ldots,\ c\omega^{|\omega|-1}\right \rbrace.\]

Using the equality of Lemma \ref{lemma:MedialGroupsEqual} implicitly, we see that per Lemma \ref{lemma:ZRegionsCosets}, the z-regions of $S_v^a$ relative to $\left \{z_v^{b_1},\ldots,\ z_v^{b_j}\right \rbrace$ correspond bijectively to the left cosets of $A'(w',G[z])$ contained in $aA(v)$.  Recall that in the derived embedding of the total graph, the vertices $w^{'x}$ and $y^{'x}$ are in the same component of the special claw corresponding to $E(d_1)$ as $v^x$.  Implicitly using the bijection of Lemma \ref{lemma:ConnectedComponentsB}, we see that the left cosets of $\langle \omega \rangle$ contained in $aA'(w')$ correspond bijectively with the 1-chains in the fiber over $z$ in $S_v^a$, and that the left cosets of $\langle \omega^2 \rangle$ contained in the left coset  $\left \{c, c\omega, c\omega^2, \ldots, c\omega^{|\omega|-1}\right \rbrace$ correspond bijectively to the two components of $R(z_v^{c})\setminus G[z]_v^c$.

Thus, using the aforementioned bijections implicitly, we may use the containment of the aforementioned left cosets to construct $\Gamma(z_v^{b_1},\ldots,\ z_v^{b_j})$: the left cosets of $A'(w')$ contained in $aA(v)$ are the vertices, the left cosets of $\langle \omega \rangle$ are the edges, and the two left cosets of $\langle \omega^2\rangle$ contained in the same left coset of $\langle \omega \rangle$ are the two ends of the same edge; an edge end is incident to the vertex that contains it. \hfill $\Box$ \end{construction}

\begin{construction}\label{construction:OrientationPreservingZGraph}  Let $z$ induce an orientation-preserving nonseparating cycle; we make no additional assumptions about $|\omega|$.  We define the set $A^\veebar(w',y',G[z])$ to be the set (not necessarily a group) of net voltages of walks in $M'$ beginning at $w'$, ending at $w'$ or $y'$, and not transversely crossing $G[z]$.  The two vertices $w^{'a}$ and $y^{'b}$ are in the same z-region, $\bar{S}_v^a$, with respect to $\left \{z_v^b:\ b \in aA(v)\right \rbrace$ if and only if there is a walk $\tilde{W}$ in the embedded $M^{'^\alpha}$ beginning at $w^{'a}$, ending at $y^{'b}$, and not transversely crossing any element of $\left \{G[z]_v^b:\ b\in aA(v)\right \rbrace$.  The last statement is true if and only if $\omega(p(\tilde{W}))=a^{-1}b$ is an element of $A^\veebar(w',y',G[z])$.  For similar reasons, $w^{'a}$ and $w^{'b}$ are contained in the same z-region if and only if $a^{-1}b$ is an element of $A^\veebar(w',y',G[z])$.  For $a_1\in A$, let $a_1\cdot A^\veebar(w',y',G[z])=\left \{a_1r: r\in A^\veebar(w',y',G[z])\right \rbrace$.  Recall that the vertices $w^{'c},\ y^{'c}$, and $v^c$ are contained in the same lift of the special claw corresponding to $E(d_1)$, and each lift of $W$ is an orientation-preserving walk.  Thus, if  $\bar{S}_{w'}^d$ contains $w^{'c}$, then it contains $w^{'c\omega^2},\ w^{'c\omega^3},\ldots,\ w^{'c\omega^{|\omega|-1}}$.  Similarly, if $\bar{S}_{w'}^d$ contains a vertex $y^{'c}$, then it contains $y^{'c\omega^2},y^{c\omega^3},\ldots,\ y^{'c\omega^{|\omega|-1}}$.  It follows that each left coset of $\langle \omega \rangle$ is contained in one or two left cosets of the form $a_1\cdot A^\veebar(w',y',G[z])$.

To any left coset $bA^\veebar(w',G[z])$ contained in $aA'(w,G[z])$, we may associate a unique set $b\cdot A^\veebar(w',y',G[z])$ that contains the group-element superscripts of the vertices in the fibers over $w'$ and $y'$ contained in $\bar{S}_{w'}^b$.  Implicitly using this bijection, the equality of Lemma \ref{lemma:MedialGroupsEqual}, and the bijections of Lemmas \ref{lemma:ConnectedComponentsC} and \ref{lemma:ZRegionsCosets}, we may say that the vertices and the edges of $\Gamma(z_v^{b_1},\ldots,z_v^{b_j})$ are the sets of the form $a_1\cdot A^\veebar(w',y',G[z])$ and the left cosets of the form $b\langle \omega \rangle$ contained in $aA(v)$, respectively; an edge is incident to the one or two vertices that contain it.\hfill $\Box$ \end{construction}

\section{Examples}\label{section:examples}

Here we give some infinite families of ordinary voltage graph embeddings whose derived embeddings have specific properties that are guaranteed by the results of the previous sections.  We show how one may use the subgroups of a voltage group $A$ developed in Section \ref{section:cosets}, and their cosets, to decompose a derived surface as a union of surfaces with boundary, whose boundary components are the fiber over a cycle in the base embedding.  Therefore, we may derive topological insight into the derived surface without constructing the entire derived embedding.  Given $a\in A$, we let $|a|$ denote $|\langle a\rangle |$ for the remainder of this section.  

Example \ref{example:SeparatingCycle} shows that for any positive integers $a$ and $b$ we can construct an ordinary voltage graph embedding containing a separating cycle $G[z]$ bounding two induced regions $S[I]$, $S[I^c]$ such that: \begin{itemize} \item{there is only one component of the derived surface, and} \item{the number of components of $S[I]^\alpha$ and $S[I^c]^\alpha$ are determined by the greatest common divisor of $a$ and $b$, and} \item{the number of cycles of $G[z]^\alpha$ bounding the components of $S[I]^\alpha$ and $S[I^c]^\alpha$ are determined by the greatest common divisor of $a$ and $b$.}   \end{itemize}

\begin{example}\label{example:SeparatingCycle}  Let $a$ and $b$ be positive integers, let $n=ab$.  Let $\mbox{lcm(a,b)}$ denote the least common multiple of $a$ and $b$, and let $\mbox{gcd(a,b)}$ denote the greatest common divisor of $a$ and $b$.  Let $c=\frac{\mbox{lcm}(a,b)}{a}$, $d=\frac{\mbox{lcm}(a,b)}{b}$, and consider Figure \ref{figure:SeparatingCycleExample}.  Let $S$ stand in the place of the sphere $S^2$.  Since $\mbox{lcm(a,b)}=\frac{ab}{\mbox{gcd(a,b)}}$, $c$ and $d$ are relatively prime, and thus $\langle c,d\rangle=\mathbb{Z}_n$.  Let $z=e_2$, and note that $z=\partial (f_1+f_2)=\partial(f_3+f_4)$.  Let $I=f_1+f_2$ and $I^c=f_3+f_4$.  Note that $A(v,S[I])=\langle d\rangle$, $A(v,S[I^c])= \langle c \rangle$, and $A(v,G[z])=\left \{ 0 \right \rbrace$.  By Part \ref{theorem:BigCosetTheoremPart1} of Theorem \ref{theorem:BigCosetTheorem}, there is only one component of $S^\alpha$.  By Part \ref{theorem:BigCosetTheoremPart2} of Theorem \ref{theorem:BigCosetTheorem}, there are $\frac{n}{|d|}$ components of $S[I]^\alpha$ contained in $S^\alpha$, and there are $\frac{n}{|c|}$ components of $S[I^c]^\alpha$ contained in $S^\alpha$.  By Part \ref{theorem:BigCosetTheoremPart3} of Theorem \ref{theorem:BigCosetTheorem}, there are $|d|$ components of $G[z]^\alpha$ contained in each component of $S[I]^\alpha$ and there are $|c|$ components of $G[z]^\alpha$ contained in each component of $S[I^c]^\alpha$.  By Construction \ref{construction:SeparatingZGraph}, we can construct $\Gamma(z_v^0,\ldots,z_v^{n-1})$ by understanding the containments of the left cosets of $A(v,G[z])$ in the left cosets of $A(v,S[I])$ and the left cosets of $A(v,S[I^c])$: the edge $z_v^r$ is incident to the vertices $S[I]_v^a$ and $S[I^c]_v^a$ if and only if $r\in \langle c\rangle \cap \langle d \rangle$.\end{example}	
	
\begin{figure}[H]
\begin{center}
\includegraphics[scale=.3]{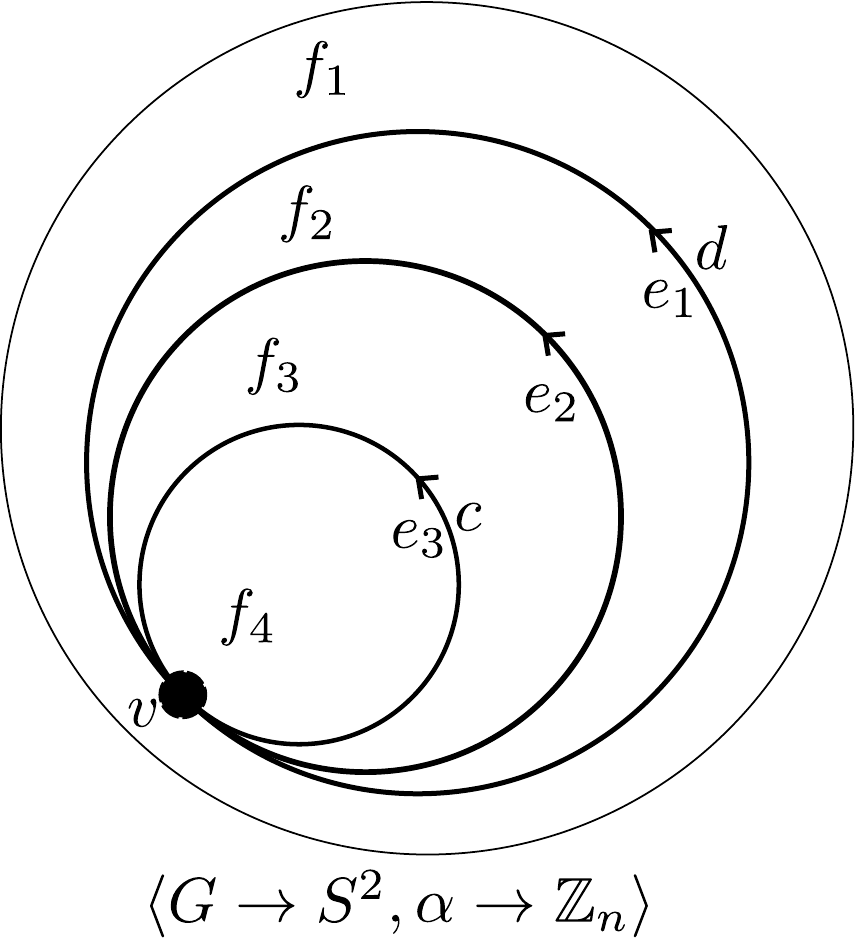}
\caption{An ordinary voltage graph embedding in the sphere $S^2$.  Only the nonidentity voltages are indicated.}\label{figure:SeparatingCycleExample}
\end{center}
\end{figure}

Example \ref{example:P2Example1} shows that for any positive integer $n$, there exists an ordinary voltage graph embedding such that: \begin{itemize} \item{there is only one component of the derived surface, and} \item{an orientation-reversing cycle $G[z]$ lifts to $n$ orientation-preserving cycles, and} \item{there are exactly two z-regions with respect to the fiber over $z$.}\end{itemize}

\begin{example}\label{example:P2Example1} Let $n$ denote any positive integer, and let and $ab$ stand for $(a,b) \in \mathbb{Z}_2 \times \mathbb{Z}_n$.  Consider Figure \ref{fig:P2Example1}.  Let $S$ stand in the place of the projective plane $P^2$.  Since there are loops in $G$ whose voltages generate $\mathbb{Z}_2\times\mathbb{Z}_n$, Part \ref{theorem:BigCosetTheoremPart1} of Theorem \ref{theorem:BigCosetTheorem} implies that there is only one component of $S^\alpha$.  Let $z\in Z(G)$ be the loop $e$, and note that $A(v,G[z])= \langle 10\rangle$.  By Part \ref{theorem:BigCosetTheoremPart3} of Theorem \ref{theorem:BigCosetTheorem}, there are $n$ cycles forming $G[z]^\alpha$.  In this case $A'(w,G[z])\cong \langle 01\rangle$.  Thus, by Theorem \ref{theorem:BigCosetTheorem2}, there are two homeomorphic z-regions of $S^\alpha$ with respect to $\left \{z_v^a:\ a\in \mathbb{Z}_2\times \mathbb{Z}_n \right \rbrace$, each bounded by the $n$ cycles forming $G[z]^\alpha$.  There are only two z-regions of $S^\alpha$ with respect to the 1-chains forming the fiber over $z$.  Per Construction \ref{construction:OrientationReversingZGraph}, we conclude that the z-graph $\Gamma(z_v^{00},z_v^{01},\ldots,z_v^{0(n-1)})$ consists of $n$ parallel edges joining two vertices.\end{example}

\begin{figure}[H]
\begin{center}
\includegraphics[scale=.4]{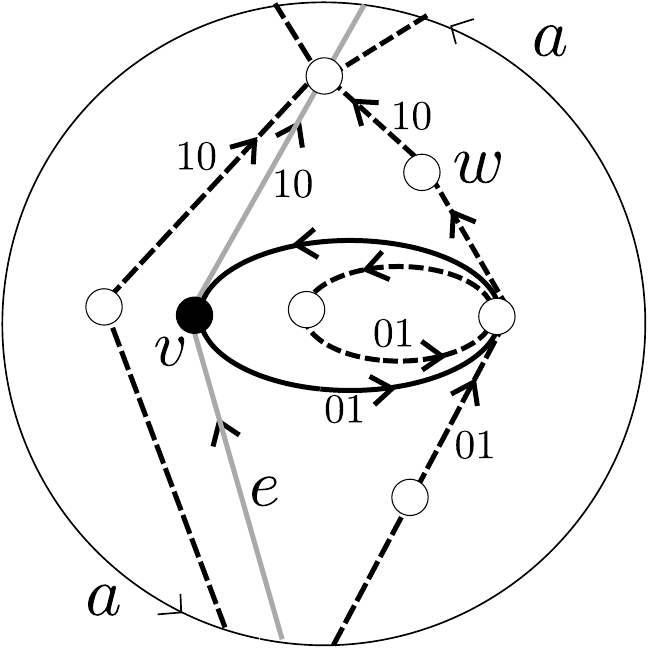}
\caption{An ordinary voltage graph embedding in the projective plane $P^2$ with the extended voltage assignment to the total graph.  We let $ab$ stand in the place of $(a,b)\in \mathbb{Z}_2\times\mathbb{Z}_n$.  Only the nonidentity voltages are indicated.  The vertices ond edges of $M'$ are white and dashed, respectively.  The edge $e$ of $G$ is colored grey.}\label{fig:P2Example1}
\end{center}
\end{figure}

Example \ref{example:P2Example2} shows that for any positive integer $n$, we can construct an ordinary voltage graph embedding such that: \begin{itemize} \item{there is only once component of the derived surface, and}  \item{an orientation-reversing cycle $G[z]$ lifts to $n$ orientation-preserving cycles, and} \item{there is only one z-region with respect to the fiber over $z$.} \end{itemize}

\begin{example}\label{example:P2Example2} Let $n$ denote any positive integer, and let $ab$ stand for $(a,b) \in \mathbb{Z}_2 \times \mathbb{Z}_n$.  Consider Figure \ref{fig:P2Example2}, and let $z\in Z(G)$ be the loop $e$ in Figure \ref{fig:P2Example2}.  Let $S$ stand in the place of $P^2$.  Since there are loops in $G$ whose voltages generate $\mathbb{Z}_2\times\mathbb{Z}_n$, Part \ref{theorem:BigCosetTheoremPart1} of Theorem \ref{theorem:BigCosetTheorem} implies that there is only one component of $S^\alpha$.  Note that $A(v,G[z])= \langle 10\rangle$.  By Part \ref{theorem:BigCosetTheoremPart3} of Theorem \ref{theorem:BigCosetTheorem}, there are $n$ cycles forming $G[z]^\alpha$.  In this case $A'(w,G[z]) = A(v)$.  Thus, by Theorem \ref{theorem:BigCosetTheorem2}, there is only one z-region of $S^\alpha$ with respect to $\left \{z_v^a:\ a\in \mathbb{Z}_2\times \mathbb{Z}_n \right \rbrace$.  Per Construction \ref{construction:OrientationReversingZGraph}, we conclude that the $\Gamma(z_v^{00},z_v^{01},\ldots,z_v^{0(n-1)})$ is a bouquet of $n$ loops.\end{example}

\begin{figure}[H]
\begin{center}
\includegraphics[scale=.4]{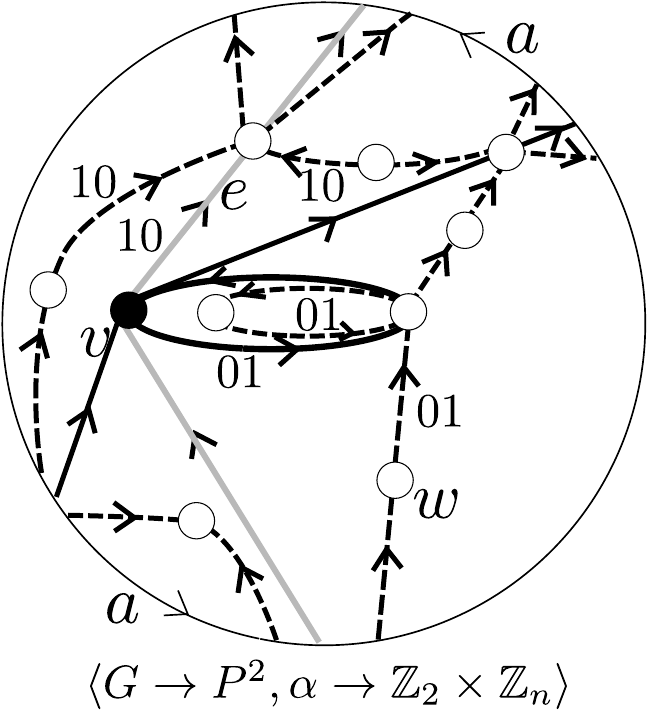}
\caption{An ordinary voltage graph embedding in the projective plane with the extended voltage assignment to the total graph.  We let $ab$ stand in the place of $(a,b)\in \mathbb{Z}_2\times\mathbb{Z}_n$.  Only the nonidentity voltages are indicated.  The vertices ond edges of $M'$ are white and dashed, respectively.  The edge $e$ of $G$ is colored grey.}\label{fig:P2Example2}
\end{center}
\end{figure}

Example \ref{example:TorusExample1} shows that for any positive integer $k$, there exists an ordinary voltage graph embedding such that the fiber over an orientation-preserving nonseparating cycle $G[z]$ bounds multiple $z$-regions with respect to the fiber over $z$, and that each of these $z$-regions is bounded by exactly $2k$ cycles of $G[z]^\alpha$.  For the remainder of this article, we let $T$ denote the torus.

\begin{example}\label{example:TorusExample1} Let integers $n$, $k$, and $d$ satisfy $n=kd$ and $n>d>1$.  Consider Figure \ref{fig:TorusExample1}, and let $z\in Z(G)$ be the loop $e$.  Let $S$ stand in the place of $T$.  Since there is a loop in $G$ with voltage $1$, we see that $A(v)=\mathbb{Z}_n$.  By Part \ref{theorem:BigCosetTheoremPart1} of Theorem \ref{theorem:BigCosetTheorem}, there is only one component of $S^\alpha$.  Since $A(v,G[z])\ = \left \{0\right \rbrace$, Part \ref{theorem:BigCosetTheoremPart3} of Theorem \ref{theorem:BigCosetTheorem} implies that there are $n$ cycles forming $G[z]^\alpha$.  Note that $A'(w,G[z])\cong \mathbb{Z}_{k}$, and so, by Theorem \ref{theorem:BigCosetTheorem2}, there are exactly $d$ z-regions of $S^\alpha$ with respect to $\left \{z_v^c:\ c\in \mathbb{Z}_n\right \rbrace$.  We conclude that each z-region is bounded by $2k$ cycles in the fiber over $G[z]$.  Per Construction \ref{construction:OrientationPreservingZGraph}, we conclude that $\Gamma(z_v^0,\ldots,z_v^{n-1})$ is a connected $2k$-regular graph consisting of $d$ vertices, and the edge $z_v^c$ is incident to the vertex $bA^\veebar(w,G[z])$ if and only if the left coset $\left \{c,\ c\omega,c\omega^2,\ldots,\ c\omega^{|\omega|-1}\right \rbrace$ of $\langle\omega\rangle$ in $A(v)$ is contained in the set $c\cdot A^\veebar(w,y,G[z])$.\end{example}

\begin{figure}[H]
\begin{center}
\includegraphics[scale=.4]{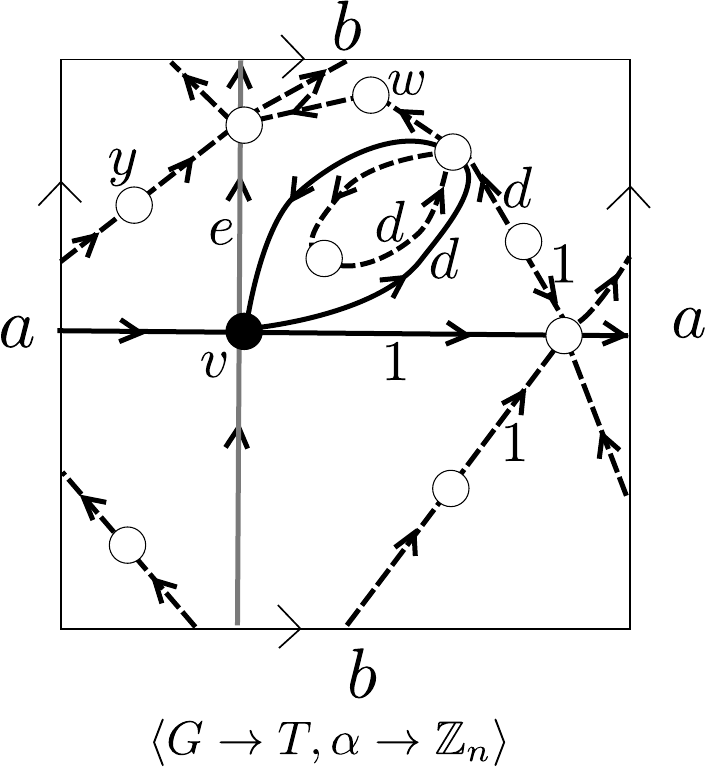}
\caption{An ordinary voltage graph embedding in the torus with the extended voltage assignment to the total graph.  Only the nonidentity voltages are indicated.  The vertices ond edges of $M'$ are white and dashed, respectively.  The edge $e$ of $G$ is colored grey.}\label{fig:TorusExample1}
\end{center}
\end{figure}

Example \ref{example:TorusExample2} shows that for any positive integer $n$, we can construct an ordinary voltage graph embedding such that: \begin{itemize} \item{there is only one component of the derived surface, and} \item{an orientation-preserving nonseparating cycle $G[z]$ lifts to $n$ orientation-preserving nonseparating cycles, and} \item{there is only one z-region with respect to the fiber over $z$.} \end{itemize}

\begin{example}\label{example:TorusExample2} Let $n$ denote any positive integer.  Consider Figure \ref{fig:TorusExample2}, and let $z\in Z(G)$ be the loop $e$.  Let $T$ denote the torus, and let $S$ stand in the place of $T$.  Since there is a loop in $G$ with voltage $1$, we see that $A(v)=\mathbb{Z}_n$, and so by Part \ref{theorem:BigCosetTheoremPart1} of Theorem \ref{theorem:BigCosetTheorem}, there is only one component of $S^\alpha$.  Since $A(v,G[z])=\left \{0\right \rbrace$, Part \ref{theorem:BigCosetTheoremPart3} implies that there are $n$ cycles forming $G[z]^\alpha$.  Note that $A'(w,G[z]) \cong \mathbb{Z}_n$, and so, by Theorem \ref{theorem:BigCosetTheorem2}, there is only one z-region of $S^\alpha$ with respect to $\left \{z_v^a:\ a\in \mathbb{Z}_n\right \rbrace$.  Per Construction \ref{construction:OrientationPreservingZGraph}, we conclude that $\Gamma(z_v^0,\ldots,z_v^{n-1})$ is a bouquet of $n$ loops.\end{example}

\begin{figure}[H]
\begin{center}
\includegraphics[scale=.4]{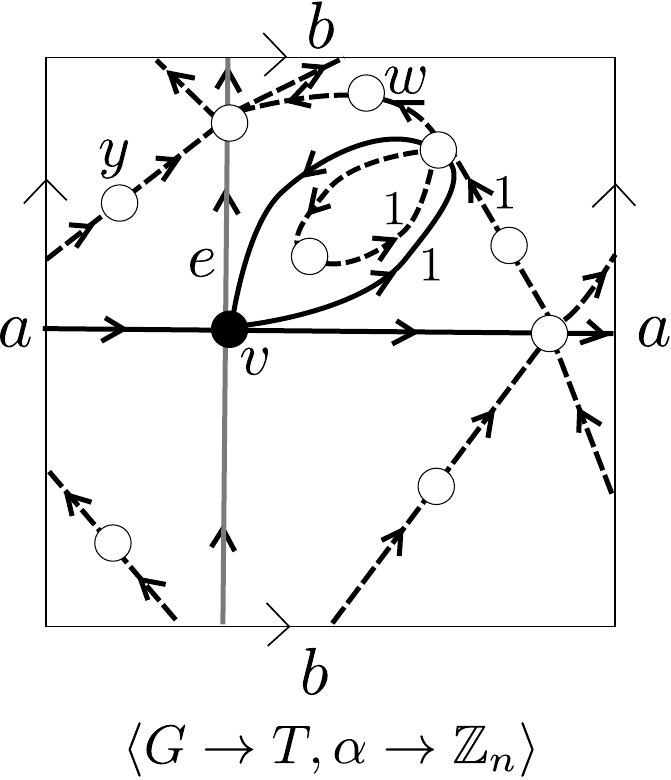}
\caption{An ordinary voltage graph embedding in the torus $T$ with the extended voltage assignment to the total graph added.  Only the nonidentity voltages are indicated.  The vertices ond edges of $M'$ are white and dashed, respectively.  The edge $e$ of $G$ is colored grey.}\label{fig:TorusExample2}
\end{center}
\end{figure}

\begin{remark}\label{remark:ConnectedComponents} Given $\langle G\rightarrow S, \alpha \rightarrow A\rangle$ and a vertex $v$ of $G$, it is possible to modify the voltage group $A$ in such a way that the derived surface has more components.  Consider the group $A\times \mathbb{Z}_n$ and the voltage assignment $\alpha_2\rightarrow A\times \mathbb{Z}_n$ satisfying $\alpha_2(e)=(\alpha(e),0)$.  If $\alpha$ satisfies $A(v)=A$, then Part \ref{theorem:BigCosetTheoremPart1} of Theorem \ref{theorem:BigCosetTheorem} implies that there is only one component of $S^\alpha$ and that there are $n$ components of $S^{\alpha_2}.$\end{remark}

\section{Acknowledgments}
The author wishes to thank his dissertation advisor, Lowell Abrams, for his wisdom and counsel.  Most of the content of this article appears in the author's dissertation.

\bibliographystyle{plain}
\bibliography{OrdinaryVoltageGraphsAndDerivedTopology.bib}

\end{document}